\documentclass[journal]{IEEEtran}

\usepackage{tikz}
\usetikzlibrary{shapes, shapes.geometric, shapes.symbols, shapes.arrows, shapes.multipart, shapes.callouts, shapes.misc}
\usepackage[activate={true,nocompatibility},final,tracking=true,kerning=true,spacing=true,factor=1100,stretch=10,shrink=10]{microtype}

\usepackage{etex}
\usepackage{graphicx}
\usepackage{amsmath}
\usepackage{amssymb}
\usepackage{mathrsfs}
\usepackage[vlined,ruled]{algorithm2e}

\usepackage[font=small]{caption}
\usepackage{subcaption}

\usepackage{cite}
\usepackage{epstopdf}
\usepackage{booktabs}

\usepackage{array}

\newtheorem{theorem}{Theorem}[section]
\newtheorem{lemma}[theorem]{Lemma}
\newtheorem{assumption}[theorem]{Assumption}
\newtheorem{definition}{Definition}
\newtheorem{corollary}[theorem]{Corollary}
\newtheorem{proposition}[theorem]{Proposition}
\newtheorem{fact}[theorem]{Fact}

\newtheorem{remark}{Remark}

\newcounter{ass}[section]
\renewcommand*{\theass}{\thesection.\arabic{ass}}

\usepackage{multirow}
\newcommand{\PQ}{\textsf{PQ} }
\newcommand{\PV}{\textsf{PV} }
\renewcommand{\j}{\boldsymbol{\mathrm{j}}}

\newenvironment{pfof}[1]{\vspace{1ex}\noindent{\itshape Proof of
    #1:}\hspace{0.5em}} {\hfill\oprocend\vspace{1ex}}
\newenvironment{proof}[1]{\vspace{1ex}\noindent{\itshape Proof:}\hspace{0.5em}} {\hfill\oprocend\vspace{1ex}}

\newcommand{\until}[1]{\{1,\dots, #1\}}

\newcommand{\setdef}[2]{\{#1 \; | \; #2\}}

\newcommand{\map}[3]{#1: #2 \rightarrow #3}

\newcommand{\real}{\mathbb{R}}

\newcommand\oprocendsymbol{\hbox{$\square$}}
\newcommand\oprocend{\relax\ifmmode\else\unskip\hfill\fi\oprocendsymbol}

\newcommand{\vect}[1]{\mathbbold{#1}}
\newcommand{\vones}[1][]{\vect{1}_{#1}}
\newcommand{\vzeros}[1][]{\vect{0}_{#1}}
\DeclareSymbolFont{bbold}{U}{bbold}{m}{n}
\DeclareSymbolFontAlphabet{\mathbbold}{bbold}

\newcommand{\tb}{\color{black}}

\renewcommand{\circle}{\mathbb{S}^1}	
\newcommand{\torus}{\mathbb{T}}

\newcommand{\bsin}{\boldsymbol{\sin}}
\newcommand{\barcsin}{\boldsymbol{\arcsin}}
\newcommand{\bcos}{\boldsymbol{\cos}}

\newcommand{\Ell}{\mathcal{E}^{\ell\ell}}
\newcommand{\Egl}{\mathcal{E}^{g\ell}}
\newcommand{\Egg}{\mathcal{E}^{gg}}

\usepackage{enumerate}
\graphicspath{{./fig/}}

\usepackage{xcolor}

\ifCLASSINFOpdf
\else
\fi
\begin{document}
%
\title{A Theory of Solvability for Lossless Power Flow Equations -- Part I: Fixed-Point Power Flow}
%
%
%

\author{John~W.~Simpson-Porco,~\IEEEmembership{Member,~IEEE}\\
\thanks{J.~W.~Simpson-Porco is with the Department of Electrical and Computer Engineering, University of Waterloo. Email: {\footnotesize jwsimpson@uwaterloo.ca}.}
}

%
%

\markboth{IEEE Transactions on Control of Network Systems. This version: \today}%
{IEEE Transactions on Control of Network Systems. This version: \today}
%



\maketitle


\vspace{-2em}

\begin{abstract}
This two-part paper details a theory of solvability for the power flow equations in lossless power networks. 
In Part I, we derive a new formulation of the lossless power flow equations, which we term the fixed-point power flow. The model is stated for both meshed and radial networks, and is parameterized by several graph-theoretic matrices -- the power network stiffness matrices -- which quantify the internal coupling strength of the network. The model leads immediately to an explicit approximation of the high-voltage power flow solution. For standard test cases, we find that iterates of the fixed-point power flow converge rapidly to the high-voltage power flow solution, with the approximate solution yielding accurate predictions near base case loading.
In Part II, we leverage the fixed-point power flow to study power flow solvability, and for radial networks we derive conditions guaranteeing the existence and uniqueness of a high-voltage power flow solution. These conditions (i) imply exponential convergence of the fixed-point power flow iteration, and (ii) properly generalize the textbook two-bus system results.
%
\end{abstract}


\section{Introduction}
\label{Section: Introduction}


The power flow equations describe the balance and flow of power in a synchronous AC power system. The solutions of these equations (also called operating points of the network) describe the configurations of voltages and currents which (i) are physically consistent with Kichhoff's and Ohm's laws, and (ii) meet the prescribed boundary conditions, specified in terms of fixed power injections or fixed voltage levels at particular nodes in the network. Knowledge of the current system operating point is crucial, as is understanding how the current operating point will change as control actions are taken or as unexpected contingencies occur. As such, the power flow equations are embedded in nearly every power system analysis or  control problem, including optimal power flow and its security-constrained variants, transient and voltage stability assessment, contingency screening, short-circuit analysis, and wide-area monitoring/control \cite{JM-JWB-JRB:08}.

As the equations are nonlinear, the existence of real-valued solutions is not guaranteed: lightly loaded networks typically possess many solutions \cite{DKM-DM-MN:16}, while a network which is loaded sufficiently will possess none. Despite this potential for both multiple reasonable solutions and infeasibility, typically there is a single desirable solution, characterized by high voltage magnitudes at buses and small inter-bus current flows. This solution is often termed stable, as it behaves in a manner consistent with the intuition of operators, and moreover, is a locally exponentially stable equilibrium point for some simplified dynamic grid models \cite{PWS-MAP:90,FD-MC-FB:11v-pnas}. The ability to accurately and consistently calculate this high-voltage solution is incredibly important, and fairly reliable numerical techniques are available for this purpose \cite{DKM-DM-MN:16,SHL:14a,SHL:14b}. 
%
%
%
While our results have computational implications, our main interest and motivation is the question of power flow feasibility/solvability: for what classes of networks and loading scenarios can we guarantee that the power flow equations are solvable for a useful solution, and what can be rigorously said about this solution?

Aside from intellectual merit, there are at least two important engineering motivations for understanding solvability. The first is to better understand the convergence of iterative numerical algorithms for solving power flow equations. {\tb When a power flow solver diverges, it may be because of a numerical instability in the algorithm, an initialization issue \cite{JST-SAN:89,JST-SAN-HDC:90}, or it may be because no power flow solution exists to be found \cite{DM-DKM-KT:16}. Without a coherent theory of power flow solvability, it is difficult to distinguish between these cases. Our proposed algorithm in Part I is based on a carefully chosen fixed-point iteration. For some restricted classes of networks, our theoretical results in Part II provide a certificate that a unique power flow solution exists, and specify a large set of initializations from which our fixed-point iteration converges exponentially to this solution.}


The second motivation comes from the desire to operate power systems safely yet non-conservatively. Due to the large capital costs of transmission infrastructure investment, system operators are incentivized to operate power networks close to their maximum power transfer limits. The present work is an additional step towards characterizing these nonlinear transfer limits, and understanding in a precise mathematical way how the transfer limits depend on the internal structure of the grid. In this context, our results in Part II provide a topology-dependent loading margin for the grid. This loading margin can serve as a solvability certificate, or as a lower bound on the distance to the maximum power transfer boundary.

\subsection{Contributions of Part I and Preview of Part II Results}
\label{Sec:Contributions}

This two-part paper presents a new model of power flow in lossless networks, and then leverages this model to obtain (i) a new iterative power flow algorithm, (ii) an approximation of the high-voltage solution, and (iii) new theoretical results on power flow solvability. Our new model is  inspired by the way that phase angles are eliminated in the standard textbook analysis of the two-bus \PV\!\!-\PQ power system \cite[Chapter 2]{TVC-CV:98}. We begin with the lossless power flow equations in polar form, with voltage variables $(V,\theta)$ and power variables $(P,Q)$, and proceed to eliminate the phase angles from the model. The state variables of our new model are (i) the normalized voltages $v_i = V_i/V_i^*$ at \PQ buses, where $V_i^*$ denotes the open-circuit voltage at the $i$th \PQ bus, and (ii) a set of slack variables which enforce Kichhoff's voltage law around cycles in meshed networks. Voltage phase angles are uniquely recovered as functions of these state variables.\footnote{In that phase angles are absent, our model is conceptually similar to the Baran-Wu branch flow model \cite{MEB-FFW:89,SHL:14a}.}

For networks without cycles, these slack variables are discarded, and the model can quickly be manipulated into the fixed-point form $v = f(v)$, where the function $f$ depends on the grid topology, parameters and loading. Motivated by this important radial case, we call our reformulation the \emph{fixed-point power flow} (FPPF). The FPPF model is the main result of Part I, and is summarized in Theorem \ref{Thm:FPPFMeshed}. In Section \ref{Sec:Linearization} we show how the FPPF model naturally leads to an explicit approximate solution to the power flow equations, yielding voltage magnitudes at \PQ buses and phase angles at all buses as explicit functions of active and reactive power injections.

{In Section \ref{Sec:Simulations} we numerically study the FPPF and the accompanying approximate solution using standard IEEE test cases. We show that the lossless power flow equations can be quickly and reliably solved by iterating the FPPF, for both lightly and heavily loaded systems. {\tb The convergence is exponential, at a rate comparable to the fast-decoupled power flow approach, and is extremely insensitive to the choice of initialization. 
We also show that our approximate solution is quite accurate in these same standard test cases.}

{\tb
Throughout Part I and Part II we restrict our attention to lossless networks, for two main reasons. Firstly, many high-voltage transmission networks are approximately lossless, with resistance/reactance ratios below $0.2$ (see Section \ref{Sec:Simulations}). For such networks, practice has shown that the lossless assumption is not restrictive. Indeed, the standard power flow model used for dispatch \textemdash{} the linearized ``DC Power Flow'' \textemdash{} explicitly relies on this lossless assumption \cite{BS-JJ-OA:09}, and is widely used in industry. Secondly, power flow solvability remains poorly understood for transmission networks, and the lossless case should be understood before attempting an analysis of the lossy case. We comment further on resistances in Section \ref{Sec:Simulations} and in our conclusions in Part II. 
}
As an informal preview of our main result in Part II, {\tb the pair of existence and uniqueness conditions we derive for radial networks with no connections between \PQ buses} are
\begin{align*}
\Delta + 4\Gamma_{g\ell}^2 &< 1\,,\qquad \Gamma_{gg} < 1\,.
\end{align*}
The first inequality captures voltage stability of \PQ load buses: $\Delta \in {[0,1)}$ is related to reactive power loading, while $\Gamma_{g\ell} \in {[0,\frac{1}{2})}$ is related to active power flow between generators and loads. Roughly speaking, this inequality ensures that voltage magnitudes at \PQ buses stay high. The second inequality on $\Gamma_{gg}$ is an angle stability condition between generator \PV buses, and ensures that phase angle differences between \PV buses stay relatively small. The quantities $\Delta, \Gamma_{g\ell}$ and $\Gamma_{gg}$ depend only on the data of the power flow problem.
These conditions are an exact generalization of conditions found in standard textbooks \cite[Chapter 2]{TVC-CV:98}, \cite[Section 8.1.1]{JM-JWB-JRB:08} for the canonical two-bus \PV\!\!-\PQ power system, generalizing the so-called circle diagram \cite[Figure 8.3]{JM-JWB-JRB:08}. They also unify and extend recent solvability conditions developed for decoupled active power flow \cite{FD-MC-FB:11v-pnas} and for decoupled reactive power flow \cite{jwsp-fd-fb:14c}, which in the above notation read as $\Gamma_{gg} < 1$ and $\Delta < 1$, respectively. {\tb We also present weaker results for networks with connections between \PQ buses, which guarantee only the existence of a solution.}  

\subsection{Structure of Paper}
\label{Sec:Structure}

Section \ref{Sec:Modeling} formally states our modeling assumptions leading to the standard model of coupled, lossless power flow used in the remainder of both papers. 

Section \ref{Sec:Reform} contains the main results of Part I. We introduce the stiffness matrices (Section \ref{Sec:Stiffness}), derive the fixed-point power flow model in a step-by-step fashion (Sections \ref{Sec:Step0}--\ref{Sec:Step3}), discuss the derivation (Section \ref{Sec:Discussion}), and derive an approximate power flow solution based on the FPPF (Section \ref{Sec:Linearization}. The derivation is presented for meshed networks, with the result for radial networks stated as a corollary. 

Section \ref{Sec:Simulations} validates our results numerically on standard test cases, while Section \ref{Sec:Conclusions} summarizes and concludes. The remainder of this section summarizes some vector and matrix notation used extensively throughout the paper, some of which is non-standard but convenient.

\subsection{Preliminaries and Notation}
\label{Sec:Notation}

{\it Sets, vectors and functions:} For a finite set $\mathcal{N}$, $|\mathcal{N}|$ is its cardinality. 
The set $\real$ (resp. $\real_{\geq 0}$, $\real_{>0}$) is the real (resp. nonnegative, strictly positive) numbers, and $\boldsymbol{\mathrm{j}}$ is the imaginary unit. For $x \in \real^{n}$ and an index set $\mathcal{I} \subset \until{n}$, $[x_i]_{i \in \mathcal{I}} \in \real^{|\mathcal{I}|\times|\mathcal{I}|}$ is the diagonal matrix containing the appropriate elements of $x$. When no confusion can arise, we will simply write $[x] \in \real^{n\times n}$ for the diagonal matrix with $x$ on the diagonal.
We let $\vones[n]$ and $\vzeros[n]$ denote the $n$-dimensional vectors of unit and zero entries, and $\vzeros[]$ is a matrix of all zeros of appropriate dimensions. The $n \times n$ identity matrix is $I_n$. The subspace $\vones[n]^{\perp} \triangleq \setdef {x \in \real^n}{\vones[n]^{\sf T}x = 0}$ is the subspace of $\real^n$ perpendicular to $\vones[n]$. 
For $x \in \real^n$, we define the vector functions $\bsin(x) \triangleq (\sin(x_1),\ldots,\sin(x_n))^{\sf T}$, with $\barcsin(x)$ and $\bcos(x)$ defined similarly, and for $x \in \real^{n}_{\geq 0}$, $\sqrt{x} = (\sqrt{x_1},\ldots,\sqrt{x_n})^{\sf T}$.

\smallskip

\emph{Graphs and graph matrices : } A graph is a pair $(\mathcal{N},\mathcal{E})$, where $\mathcal{N}$ is the set of nodes and $\mathcal{E} \subseteq \mathcal{N} \times \mathcal{N}$ is the set of edges.
If a label $e \in \until{|\mathcal{E}|}$ and an arbitrary direction is assigned to each edge $e = (i,j) \in \mathcal{E}$, the \emph{node-edge incidence matrix} $A \in \real^{|\mathcal{N}|\times|\mathcal{E}|}$ is defined component-wise as $A_{ke} = 1$ if node $k$ is the source node of edge $e$ and as $A_{ke} = -1$ if node $k$ is the sink node of edge $e$, with all other elements being zero. A graph is \emph{radial} (\emph{acyclic, a tree}) if it contains no cycles.
For $x \in \real^{|\mathcal{N}|}$, $A^{\sf T}x \in \real^{|\mathcal{E}|}$ is the vector with components $x_i-x_j$, with $(i,j) \in \mathcal{E}$. 
We call $\mathrm{ker}(A)$ the \emph{cycle space} of the graph. If the graph is connected, then $\mathrm{ker}(A^{\sf T})=\mathrm{im}(\vones[|\mathcal{N}|])$, with $\mathrm{ker}(A) =
\emptyset$\@ for acyclic graphs. In this radial case, for every $P \in
\vones[|\mathcal{N}|]^\perp$, there exists a unique $p \in
\real^{|\mathcal{E}|}$ satisfying Kirchoff's Current Law (KCL) $P = Ap$
\cite{NB:97,LOC-CAD-EDK:87}. The vector $P$ is interpreted as nodal
injections, with $p$ being the associated edge flows. If a weight $w_{ij} > 0$ is assigned to each edge $(i,j)\in\mathcal{E}$, then ${\sf L} = {\sf L}^{\sf T} = A[w]A^{\sf T}$ is the weighted \emph{Laplacian matrix} for the graph, which is positive semidefinite with $\mathrm{ker}({\sf L}) = \mathrm{im}(\vones[|\mathcal{N}|])$.

\smallskip

\emph{Torus geometry:} The set $\circle$ denotes the \emph{unit circle}, an \emph{angle} is a point $\theta \in \circle$, and $|\theta_1 - \theta_2|$ denotes the \emph{geodesic distance} between two angles $\theta_1,\theta_2 \in \circle$. The \emph{$n$-torus} $\torus^n = \circle \times \cdots \times \circle$ is the Cartesian product of $n$ unit circles.
For $\gamma \in [0,\pi/2)$ and a given graph $(\mathcal{N},\mathcal{E})$, let $\Theta(\gamma) =\{ \theta \in \mathbb T^{|\mathcal{N}|}:\, \max_{(i,j) \in \mathcal E} |\theta_{i} - \theta_{j}| \leq \gamma \}$ be the closed set of angles $\theta = ( \theta_{1},\dots,\theta_{n})$ with neighboring angles $\theta_{i}$ and $\theta_{j}$ no further than $\gamma$ apart. 

\smallskip

\section{The Lossless Power Flow Model}
\label{Sec:Modeling}

This section introduces the power flow model used throughout the paper and states all modeling assumptions.

\subsection{Network and Branch Models}
\label{Sec:NetworkBranchModel}

We model a steady-state synchronous power network as a connected, weighted, and undirected graph $(\mathcal{N},\mathcal{E})$ with nodes (buses) $\mathcal{N}$ and edges (branches) $\mathcal{E} \subseteq \mathcal{N} \times \mathcal{N}$. To each bus $i \in \mathcal{N}$ we associate a phasor voltage $V_i\angle \theta_i$, where $V_i > 0$ is the bus voltage magnitude and $\theta_i \in \circle$ is the voltage phase angle, and a complex power injection $P_i + \j Q_i$. The real part $P_i$ is the \emph{active} power  injection, while $Q_i$ is the \emph{reactive} power injection. There will be two types of buses in our network: we will have $n \geq 1$ \emph{load buses} buses, denoted by the set $\mathcal{N}_L$, and $m \geq 1$  \emph{generator buses}, denoted by the set $\mathcal{N}_G$. Without loss of generality, we order these buses as $\mathcal{N}_L = \{1,\ldots,n\}$ and $\mathcal{N}_G = \{n+1,\ldots,n+m\}$, with $\mathcal{N} = \mathcal{N}_L\cup\mathcal{N}_G$. Models for these buses are stated in Section \ref{Sec:BusModel}.
This partitioning of buses $\mathcal{N} = \mathcal{N}_L\cup\mathcal{N}_G$ induces a partitioning of the branches\footnote{To keep notation under control, we will abbreviate $\vones[|\Ell|]$, $\vzeros[|\Ell|], I_{|\Ell|}$ by $\vones[\ell\ell], \vzeros[\ell\ell], I_{\ell\ell}$, and similarly for the other edge sets.}
\begin{equation}\label{Eq:Edges}
\mathcal{E} = \Ell \cup \Egl \cup \Egg\,,
\end{equation}
where, $\Ell$ contains all branches between nodes $i,j \in \mathcal{N}_L$, $\Egl$ contains all branches between generators and loads, etc. The incidence matrix $A \in \real^{(n+m)\times|\mathcal{E}|}$ of the graph $(\mathcal{N},\mathcal{E})$ inherits both the nodal and branch partitions, and may be written as the 2 $\times$ 3 block matrix
\begin{equation}\label{Eq:Incidence}
A = \left(\begin{array}{c}A_L \\ [0.5ex] \hline \\ [-2ex] A_G\end{array}\right) = 
\left(
\begin{array}{c|c|c}
A_L^{\ell\ell} & A_L^{g \ell} & \vzeros[] \\ [0.5ex] \hline \\ [-2ex] \vzeros[] & A_G^{g \ell} & A_G^{g g}
\end{array}
\right)\,.
\end{equation}
For example, $A_{L}^{\ell\ell}$ is a matrix of size $n \times |\Ell|$, mapping variables defined on the branches between load buses to variables defined at the load buses incident to those branches. The zero submatrices in \eqref{Eq:Incidence} indicate that branches between generators cannot be incident to any load buses, and vice versa. Since the incidence matrix assigns an arbitrary orientation to each branch, we assume without loss of generality that all branches in $\mathcal{E}^{g\ell}$ are oriented from generators to loads: for each $(i,j) \in \mathcal{E}^{g\ell}$, $i \in \mathcal{N}_G$ and $j \in \mathcal{N}_L$. It follows that all non-zero elements of ${A}^{g\ell}_L$ are equal to $-1$, while all non-zero elements of $A^{g\ell}_G$ are equal to $+1$. Figure \ref{Fig:Notation} illustrates these conventions.

\begin{figure}[ht!]
\begin{center}
\includegraphics[width=1\columnwidth]{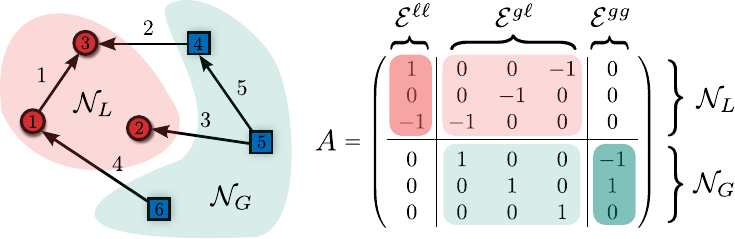}
\caption{Example network showing division of buses and edges with $|\mathcal{N}_G| = 3$ blue \PV buses, $|\mathcal{N}_L| = 3$ red \PQ buses, $|\Ell| = 1$, $|\Egl| = 3$, and $|\Egg| = 1$. Edges $(i,j) \in \mathcal{E}^{g\ell}$ are oriented from \PV buses to \PQ buses, while the orientation of other edges is arbitrary.}
\label{Fig:Notation}
\end{center}
\end{figure}

We assume all transmission lines are inductive, and model them with the standard lumped parameter ${\Pi}$-model, which allows for the inclusion of inductive/capactive shunt loads, off-nomial tap ratios, and line-charging capacitors \cite{PK:94}. With this model, each branch $(i,j) \in \mathcal{E}$ is weighted by a purely imaginary admittance $y_{ij} = \j b_{ij}$ where $b_{ij} \leq 0$.

We encode the admittances and grid topology in the bus admittance matrix $Y \in \real^{(n+m)\times(n+m)}$, with elements $Y_{ij} = -y_{ij}$ and $Y_{ii} = \sum_{j\neq i}^{n+m}y_{ij} + y_{\mathrm{shunt},i}$, where $y_{\mathrm{shunt},i} = \j b_{\mathrm{shunt},i}$ is the shunt element at bus $i$. The susceptance $b_{\mathrm{shunt},i}$ and can be positive (capacitive) or negative (inductive). The susceptance matrix $B \in \real^{(n+m)\times(n+m)}$ is defined as the imaginary part of $Y$, and is characterized as follows \cite{RK-FFW:84}.

\medskip

\begin{fact}[\bf Properties of the Susceptance Matrix]\label{Fact:Susceptance}
If the network contains no phase-shifting transformers, then
\begin{enumerate}[(i)]\setlength{\itemsep}{1.5pt}
\item $B_{ij} = B_{ji} \geq 0$, with $B_{ij} > 0$ if and only if $(i,j) \in \mathcal{E}$ or $(j,i)\in\mathcal{E}$;
\item $B_{ii} = -\sum_{j=1,j\neq i}^{n+m} B_{ij} + b_{\mathrm{shunt},i}$ for all $i \in \mathcal{N}$.
\end{enumerate}
\end{fact}
\medskip
We vectorized the relevant variables as $V = (V_1,\ldots,V_{n+m})^{\sf T}$, $\theta = (\theta_1,\ldots,\theta_{n+m})^{\sf T}$, $P = (P_1,\ldots,P_{n+m})^{\sf T}$, and $Q = (Q_1,\ldots,Q_{n+m})^{\sf T}$. Like the incidence matrix, these vectors and the susceptance matrix $B$ inherit the partitioning of buses $\mathcal{N} = \mathcal{N}_L\cup\mathcal{N}_G$ as
\begin{equation}\label{Eq:Susceptance}
V = \begin{pmatrix} V_L \\ V_G \end{pmatrix}\,,\quad Q = \begin{pmatrix}Q_L \\ Q_G \end{pmatrix}\,,\quad B = \begin{pmatrix}
B_{LL} & B_{LG}\\
B_{GL} & B_{GG}
\end{pmatrix}\,.
\end{equation}
The principal submatrix $B_{LL} \in \real^{n \times n}$ describes the weighted interconnections between \PQ buses and is central to our analysis; we refer to it as the \emph{grounded susceptance matrix}, and impose the following standing assumption on it.

\smallskip

\begin{assumption}[\bf Grounded Susceptance Matrix]\label{Ass:Matrix}
The grounded susceptance matrix $B_{LL}$ in \eqref{Eq:Susceptance} is negative definite.
\end{assumption}

\smallskip

Assumption \ref{Ass:Matrix} is usually satisfied in practical networks \cite[Section III]{JT-DS-MIS:86}, and always satisfied in the absence of line-charging and shunt capacitors due to irreducibility and diagonal dominance of the susceptance matrix \cite{RAH-CRJ:85}. This assumption does not disallow shunt capacitors, but merely limits their size so that they do not ``overcompensate'' the network.

\subsection{Bus Models}
\label{Sec:BusModel}

\emph{Load Models:} Each load bus $i \in \mathcal{N}_L$ is modeled as a \PQ bus, with a fixed active power injection $P_i \in \real$ and a fixed reactive power injection $Q_i \in \real$. At \PQ buses, voltage magnitudes $V_i > 0$ and phases $\theta_i \in \circle$ are free variables. 

\emph{Generator Models:} Each generator bus $i \in \mathcal{N}_G$ is modeled as a \PV bus, with an active power injection $P_i \in \real$ fixed by the prime mover and constant voltage magnitude $V_i > 0$ regulated on the network-side by an Automatic Voltage Regulator (AVR) system. At \PV buses, phase angles $\theta_i \in \circle$ and reactive power injections $Q_i \in \real$ are free variables. {\tb In the present work we do not consider generator reactive power constraints.} We will therefore generally be unconcerned with reactive power injections at \PV buses, as they are {\tb uniquely} determined by \eqref{Eq:Reactive} given the other state variables. In other words, reactive power injections at \PV buses may be considered as ``outputs'' of the power flow problem, rather than state variables. 

The above bus models are uncontroversial when considering power flow solvability, and are the standard models \cite{JM-JWB-JRB:08} used by industry \cite{JVM-KY-SMV-SZD-LMK:13}. Since the network is lossless, no ``slack bus'' is required; all generators are treated as \PV buses.\footnote{Equivalently, select \PV bus $(n+m)$ as the slack bus, and its power injection $P_{n+m} = -\sum_{i=1}^{n+m-1} P_i$ is determined a priori.}

\subsection{The Power Flow Equations}
\label{Sec:PowerFlowEquations}

Denoting the voltage phasor at bus $i \in \mathcal{N}$ by $U_i = V_ie^{\j\theta_i}$, the bus current injections $I \in \mathbb{C}^{n+m}$ in the network are given by the so-called nodal equations $I = YU = \j BU$. The equations for balance of power may then be written as $P + \j Q = [U]\mathrm{conj}(I^*)$, where $\mathrm{conj}(\cdot)$ denotes complex conjugation. After expanding the complex exponentials into trigonometric terms  and equating real and imaginary parts, one arrives at the celebrated lossless power flow equations
\begin{subequations}
\begin{align}\label{Eq:Active}
P_i &= \sum_{j=1}^{n+m} \nolimits V_iV_jB_{ij}\sin(\theta_i-\theta_j)\,,\quad i \in \mathcal{N}_L \cup \mathcal{N}_G\,,\\\label{Eq:Reactive}
Q_i &= -\sum_{j=1}^{n+m} \nolimits V_iV_jB_{ij}\cos(\theta_i-\theta_j)\,,\quad i \in \mathcal{N}_L\,,
\end{align}
\end{subequations}
written here in polar form, and where we have suppressed the reactive power equations for the \PV buses (Section \ref{Sec:BusModel}). The equations \eqref{Eq:Active}--\eqref{Eq:Reactive} are a set of $2n+m$ nonlinear equations in the $n+m$ angles $\theta$ and the $n$ \PQ bus voltage magnitudes $V_L$. As the angles enter only as differences, there are effectively only $n+m-1$ angles, and the problem appears to be overdetermined. However, note from \eqref{Eq:Active} that $\sum_{i=1}^{n+m} P_i = 0$ for all choices of angles and voltage magnitudes; this equation states the balance of active power in a lossless network. In other words, \eqref{Eq:Active} is not an independent set of equations. Rather than eliminate one angle and one active power flow equation, we will retain both and impose the following assumption on the active power injections $P = (P_1,\ldots,P_{n+m})^{\sf T}$.

\smallskip

\begin{assumption}[\bf Balance Assumption]\label{Ass:Balance}
The necessary condition $P \in \vones[n+m]^{\perp}$ is satisfied.
\end{assumption}

\section{Reformulation of the Power Flow Equations to Fixed-Point Power Flow}
\label{Sec:Reform}

We now begin the formulation of our new fixed-point power flow model. Several key matrices \textemdash{} which we term the power network stiffness matrices \textemdash{} will appear very naturally during this reformulation process, and we now introduce the reader to them. These matrices quantify the strength of the grid, and are analogous to the stiffness matrices encountered in the theories of mechanical statics and vibrations.

\subsection{Open-Circuit Voltages and Stiffness Matrices}
\label{Sec:Stiffness}

Letting $V_G = (V_{n+1},\ldots,V_{n+m})^{\sf T}$ denote the vector of generator voltage magnitudes, we introduce the following definition and characterization.

\smallskip

\begin{definition}[\bf Open-Circuit Load Voltages]\label{Def:VLstar}
The \emph{open-circuit load voltages} $V_L^* \in \real^{n}$ are defined using the susceptance matrix \eqref{Eq:Susceptance} as
\begin{equation}\label{Eq:VLstar}
V_L^* \triangleq -B_{LL}^{-1}B_{LG}V_G\,.
\end{equation}
\end{definition}

\begin{proposition}[\bf Open-Circuit Solution]\label{Prop:VLstar}
Each component of $V_L^*$ is strictly positive. Moreover, when $P = \vzeros[n+m]$ and $Q_L = \vzeros[n]$, $(\theta,V_L) = (\vzeros[n+m],V_L^*)$ is a solution of \eqref{Eq:Active}--\eqref{Eq:Reactive}.
\end{proposition}

\begin{proof}. See appendix. \end{proof}

\smallskip

When there are no shunt elements attached at \PQ buses, the matrix $-B_{LL}^{-1}B_{LG}$ is a nonnegative row-stochastic matrix \cite[Lemma II.1]{FD-FB:11d}, and each open-circuit load voltage $V_i^*$ is therefore a weighted average of \PV bus voltage set points. Capacitive shunt elements tend to push these open-circuit voltages up, while inductive shunt elements pull them down.

We now use these open-circuit voltages to define three matrices of interest, which have units of power.

\smallskip

\begin{definition}\label{Def:Stiffness}\textbf{(Laplacian, Branch, and Nodal Stiffness Matrices):}
The \emph{branch stiffness matrix} $\mathsf{D} \in \real^{|\mathcal{E}|\times|\mathcal{E}|}$ is the diagonal positive definite matrix\footnote{For convenience we abuse our notation a bit, and here use $V_i^*$ also for the fixed \PV bus voltages.}
\begin{equation}\label{Eq:DMatrix}
\mathsf{D} \triangleq [V_i^*V_j^*B_{ij}]_{(i,j)\in\mathcal{E}}\,.
\end{equation}
The \emph{Laplacian stiffness matrix} ${\sf L} = {\sf L}^{\sf T} \in \real^{(n+m)\times(n+m)}$ is the symmetric positive semidefinite matrix
\begin{equation}\label{Eq:L}
{\sf L} \triangleq A\mathsf{D}A^{\sf T}\,.
\end{equation}
With the partitioning of the susceptance matrix $B$ in \eqref{Eq:Susceptance}, the \emph{nodal stiffness matrix} $\mathsf{S} \in \real^{n \times n}$ is the symmetric, negative definite Metzler matrix
\begin{equation}\label{Eq:Qcrit}
\mathsf{S} \triangleq \frac{1}{4}[V_L^*]B_{LL}[V_L^*]\,.
\end{equation}
\end{definition}

\smallskip

To interpret ${\sf D}$, note from \eqref{Eq:Active} that $V_iV_jB_{ij}$ is the maximum active power transfer along the branch $(i,j) \in \mathcal{E}$. Thus, the branch stiffness matrix \eqref{Eq:DMatrix} captures the maximum branch-wise power transfers when voltages $V_i$ take their open-circuit values $V_i^*$. The Laplacian stiffness matrix \eqref{Eq:L}, first introduced in \cite{FD-MC-FB:11v-pnas}, is a generalization of the branch matrix $\mathsf{D}$ for meshed networks. The nodal stiffness matrix was introduced by the author in \cite{jwsp-fd-fb:14c}.
\emph{Roughly speaking, $\mathsf{D}$ and ${\sf L}$ quantify how strong the branches of the network are, while $\mathsf{S}$ quantifies the interconnection strength between \PQ buses.}
%
%
{All three matrices depend on the open-circuit voltages \eqref{Eq:VLstar}, which in turn depend quite densely on the interconnection structure of the network, including shunt elements at \PQ buses.}
The branch stiffness matrix $\mathsf{D}$ inherits the partitioning of the branches $\mathcal{E}$ in \eqref{Eq:Edges} as
\begin{equation}\label{Eq:DMatrix2}
{\sf D} = \mathrm{blkdiag}(\mathsf{D}_{\ell\ell}\,, \mathsf{D}_{g\ell}\,, \mathsf{D}_{gg})\,.
\end{equation}

\subsection{Procedure for Deriving Fixed-Point Power Flow}
\label{Sec:Step0}

The first major challenge one encounters in analyzing \eqref{Eq:Active}-\eqref{Eq:Reactive} is that variables appear as complicated products of trigonometric nonlinearities $\sin(\cdot)$ and $\cos(\cdot)$ with quadratic nonlinearities $V_iV_j$. More precisely, for any branch $(i,j)\in\mathcal{E}$ the product $V_iV_j$ will be a quadratic nonlinearity if $i$ and $j$ are \PQ buses, linear if $i$ and $j$ are a \PQ\!\!-\PV pair, or a constant if both buses are \PV buses. We must therefore first figure out how to isolate these different combinations in a useful way. After this initial hurdle, our reformulation procedure will consist of three main steps:

\begin{enumerate}\setlength{\itemsep}{1.5pt}
\item[] \textbf{Step 1:} Isolate the phases in the active power flow \eqref{Eq:Active};
\item[] \textbf{Step 2:} Reformulate the $n$ reactive power flow equations \eqref{Eq:Reactive} into a form which partially isolates the phase angles;
\item[] \textbf{Step 3:} {\tb Combine the two reformulations by eliminating the phase angles from the reformulated reactive power flow, resulting in a single equation in the voltage variables, and rearrange the result into a fixed-point equation.}
\end{enumerate}

\subsection{Reformulation Step 1: Active Power Flow}
\label{Sec:Step1}

We begin by introducing some additional notation associated with the incidence matrix \eqref{Eq:Incidence}. Due to the way (see Section \ref{Sec:NetworkBranchModel}) that edge directions were assigned for the graph $(\mathcal{N},\mathcal{E})$, we may write $A$ as the difference between two nonnegative matrices $A = A(+) - A(-)$, i.e.,
\begin{equation}\label{Eq:IncidencePlusMinus}
\begin{aligned}
A &= 
\begin{pmatrix}
A_L^{\ell\ell}(+) & \vzeros[] & \vzeros[]\\
\vzeros[] & A_G^{g \ell}(+) & A_G^{gg}(+)
\end{pmatrix}\\
&\quad -
\begin{pmatrix}
A_L^{\ell\ell}(-) & A_L^{g \ell}(-) & \vzeros[]\\
\vzeros[] & \vzeros[] & A_G^{gg}(-)
\end{pmatrix}\,.
\end{aligned}
\end{equation} 
The matrix $A(+)$ indexes the buses at the \emph{sending end} of each branch, while $A(-)$ indexes the corresponding \emph{receiving end} buses. It follows that $A(+)^{\sf T}V$ is vector of sending-end voltages, while $A(-)^{\sf T}V$ is the vector of receiving end voltages. In vector form, we compute that
\begin{subequations}
\begin{align}\label{Eq:EdgeVoltage1}
A(+)^{\sf T}V = A(+)^{\sf T} \begin{pmatrix}V_L \\ V_G \end{pmatrix} &
= 
\begin{pmatrix}
A_L^{\ell\ell}(+)^{\sf T}V_L \\ 
A_G^{g \ell}(+)^{\sf T} V_G\\
A_G^{gg}(+)^{\sf T} V_G
\end{pmatrix}\,,\\
\label{Eq:EdgeVoltage2}
A(-)^{\sf T}V = A(-)^{\sf T} \begin{pmatrix}V_L \\ V_G \end{pmatrix} &= 
\begin{pmatrix}
A_L^{\ell\ell}(-)^{\sf T}V_L \\ 
A_L^{g \ell}(-)^{\sf T} V_L\\
A_G^{gg}(-)^{\sf T} V_G
\end{pmatrix}\,.
\end{align}
\end{subequations}
For $(i,j) \in \mathcal{E}$, the quadratic products $V_iV_jB_{ij}$ can therefore be constructed from \eqref{Eq:EdgeVoltage1}--\eqref{Eq:EdgeVoltage2} through the formula
\begin{align}\nonumber
&[V_iV_jB_{ij}]_{(i,j)\in\mathcal{E}}\\ \vspace{1em} 
\label{Eq:Step1Temp1}
&= \left[\begin{pmatrix}
A_L^{\ell\ell}(+)^{\sf T}V_L \\ 
A_G^{g \ell}(+)^{\sf T} V_G\\
A_G^{gg}(+)^{\sf T} V_G
\end{pmatrix}\right] \left[\begin{pmatrix}
A_L^{\ell\ell}(-)^{\sf T}V_L \\ 
A_L^{g \ell}(-)^{\sf T} V_L\\
A_G^{gg}(-)^{\sf T} V_G
\end{pmatrix}\right]  [B_{ij}]_{(i,j)\in\mathcal{E}}\,.
\end{align}
It follows from \eqref{Eq:Step1Temp1} that the branch stiffness matrix $\mathsf{D} = [V_i^*V_j^*B_{ij}]_{(i,j)\in\mathcal{E}}$ in \eqref{Eq:DMatrix} has the vector representation
\begin{equation}\label{Eq:DMatrix4}
\mathsf{D} = \left[\begin{pmatrix}
A_L^{\ell\ell}(+)^{\sf T}V_L^* \\ 
A_G^{g \ell}(+)^{\sf T} V_G\\
A_G^{gg}(+)^{\sf T} V_G
\end{pmatrix}\right] \left[\begin{pmatrix}
A_L^{\ell\ell}(-)^{\sf T}V_L^* \\ 
A_L^{g \ell}(-)^{\sf T} V_L^*\\
A_G^{gg}(-)^{\sf T} V_G
\end{pmatrix}\right]  [B_{ij}]\,,
\end{equation}
and the submatrices in \eqref{Eq:DMatrix2} may be easily identified by comparing to \eqref{Eq:DMatrix4}. We now find it useful to introduce a change of voltage variables. Using the open-circuit load voltages $V_L^*$ defined in \eqref{Eq:VLstar}, consider the bijective change of variables
\begin{equation}\label{Eq:x}
V_L = [V_L^*]v \quad \Longleftrightarrow \quad v = [V_L^*]^{-1}V_L \,.
\end{equation}
Thus, $v_i$ is the load bus voltage $V_i$ normalized by its open-circuit value $V_i^*$. Inserting the coordinate change \eqref{Eq:x} into \eqref{Eq:Step1Temp1}, we find that
\begin{equation}\label{Eq:DMatrix5}
[V_iV_jB_{ij}]_{(i,j)\in\mathcal{E}} = \mathsf{D}\cdot[h(v)]\,,
\end{equation}
where the map $\map{h}{\real^n}{\real^{|\mathcal{E}|}}$ satisfies $h(\vones[n]) = \vones[|\mathcal{E}|]$ and is defined by
\begin{equation}\label{Eq:VVectorForm}
\begin{aligned}
&h(v) = \begin{pmatrix}h_{\ell\ell}(v) \\ h_{g\ell}(v) \\ h_{gg}(v)\end{pmatrix}
&= 
\begin{pmatrix}
[A_L^{\ell\ell}(+)^{\sf T}v](A_L^{\ell\ell}(-)^{\sf T}v)\\
A_L^{g\ell}(-)^{\sf T}v\\
\vones[gg]\\
\end{pmatrix}\,.
\end{aligned}
\end{equation}
More explicitly, for any edge $e = (i,j) \in \mathcal{E}$, $h_{e}(v)$ satisfies
\begin{equation}\label{Eq:VComponentForm}
h_{e}(v) =
  \begin{cases}
   v_iv_j & \text{if } e=(i,j) \in \mathcal{E}^{\ell \ell}\,\\
   v_j       & \text{if } e=(i,j) \in \mathcal{E}^{g \ell}\,\\
   1 & \text{if } e=(i,j) \in \mathcal{E}^{gg}\,
  \end{cases}\,.
\end{equation}
\smallskip
In equation \eqref{Eq:DMatrix5}, we have decomposed the products $V_iV_jB_{ij}$ into a product of the branch stiffness matrix $\mathsf{D}$ and a nonlinear function $h(v)$. We now return to the active power flow \eqref{Eq:Active}.
In vector form, \eqref{Eq:Active} is written as
\begin{equation}\label{Eq:ActivePowerVector}
P = A\,[V_iV_jB_{ij}]_{(i,j)\in\mathcal{E}}\,\bsin(A^{\sf T}\theta)\,,
\end{equation}
as may be verified by applying the definition of the incidence matrix. Substituting our result \eqref{Eq:DMatrix5} for $[V_iV_jB_{ij}]_{(i,j)\in\mathcal{E}}$ into the active power flow \eqref{Eq:ActivePowerVector} we obtain
\begin{equation}\label{Eq:ActivePowerD}
P = A\,\mathsf{D}\cdot[h(v)]\cdot \bsin(A^{\sf T}\theta)\,.
\end{equation}
Let $c \triangleq \mathrm{dim}(\mathrm{ker}(A))$ be the dimension of the cycle space for the graph $(\mathcal{N},\mathcal{E})$. In other words, $c$ is the number of independent cycles. Let $C \in \real^{|\mathcal{E}|\times c}$ be a matrix whose columns span $\mathrm{ker}(A)$.%
\footnote{{\tb One may always find a so-called totally unimodular basis for the cycle space $\mathrm{ker}(A)$, in which case one may take $C \in \{-1,0,1\}^{|\mathcal{E}|\times c}$ as the corresponding \emph{(oriented) edge-cycle incidence matrix} \cite[Section 3]{TK-CL-KM-DM-RR-TU-KAZ:09}. In this case, $C$ is totally unimodular \cite[Theorem 3.4]{TK-CL-KM-DM-RR-TU-KAZ:09}; we proceed with this choice.}} 
Inspired by \cite[SI Theorem 1]{FD-MC-FB:11v-pnas}, the following result shows that by using the Laplacian stiffness matrix ${\sf L}$, equation \eqref{Eq:ActivePowerD} can be solved for $\bsin(A^{\sf T}\theta)$ in terms of the voltages $v$ and a set of slack variables $y \in \real^c$ which account for loop flows of active power.

\smallskip

\begin{lemma}\label{Lem:ActiveMesh}\textbf{(Active Power Reformulation for Meshed Networks):}
Consider the active power flow equation \eqref{Eq:Active}, and let $\map{\psi}{\real_{>0}^n \times \real^c}{\real^{|\mathcal{E}|}}$ be defined by\footnote{Here, ${\sf L}^{\dagger}$ denotes the Moore-Penrose pseudoinverse of ${\sf L}$.}
\begin{equation}\label{Eq:MeshSolution}
\psi(v,y) \triangleq [h(v)]^{-1}\left(A^{\sf T}{\sf L}^{\dagger}P + \mathsf{D}^{-1}Cy\right)\,.
\end{equation}
The following two statements are equivalent:
\begin{enumerate}\setlength{\itemsep}{1.5pt}
\item[(i)] $(\theta,V_L) \in \Theta(\pi/2) \times \real^{n}_{>0}$ is a solution of \eqref{Eq:Active};
\item[(ii)] $(v,y) \in \real_{>0}^n \times \real^{c}$ is a solution of
\begin{equation}\label{Eq:CycleConstraint}
\vzeros[c] = C^{\sf T}\barcsin(\psi(v,y)) \qquad {\tb (\boldsymbol{\mathrm{mod}}\,2\pi)}\,.
\end{equation}
with branch-wise angle differences $\eta = A^{\sf T}\theta\,\, {\tb (\boldsymbol{\mathrm{mod}}\,2\pi)} \in [-\frac{\pi}{2},\frac{\pi}{2}]^{|\mathcal{E}|}$ recovered via
\begin{equation}\label{Eq:Candidate}
\bsin(\eta) = \psi(v,y)\,.
\end{equation}
\end{enumerate}
\end{lemma}

\begin{proof}.
Our development so far shows that \eqref{Eq:Active} is equivalent to \eqref{Eq:ActivePowerD}, so we show equivalence of \eqref{Eq:ActivePowerD} and \eqref{Eq:CycleConstraint}--\eqref{Eq:Candidate}. First, simply make a change of variables $\tilde{\psi} = \bsin(A^{\sf T}\theta)$. Then \eqref{Eq:ActivePowerD} may be written as the pair of equations
\begin{subequations}
\begin{align}
\label{Eq:ActiveExpanded1}
P &= A\,\mathsf{D}\cdot [h(v)] \cdot \tilde{\psi}\\
\label{Eq:ActiveExpanded2}
\tilde{\psi} &= \bsin(A^{\sf T}\theta)\,.
\end{align}
\end{subequations}
The equation \eqref{Eq:ActiveExpanded1} is linear in $\tilde{\psi}$, and we claim that $\psi(v,y)$ as defined in \eqref{Eq:MeshSolution} is the general solution. Substituting \eqref{Eq:MeshSolution} into \eqref{Eq:ActiveExpanded1}, we indeed find that
\begin{align*}
P &= A\mathsf{D}\left(A^{\sf T}{\sf L}^{\dagger}P+\mathsf{D}^{-1}Cy\right)\\
&= \underbrace{A\mathsf{D}A^{\sf T}}_{={\sf L}}{\sf L}^{\dagger}P + A\mathsf{D}\mathsf{D}^{-1}Cy
= \Pi P + \underbrace{AC}_{=\vzeros[]}y = P\,,
\end{align*}
where ${\sf L}{\sf L}^{\dagger} = \Pi$ is the projection matrix onto $\vones[n+m]^{\perp}$, and we used the facts that $P \in \vones[n+m]^{\perp}$ and that $AC = \vzeros[]$ by construction. Thus \eqref{Eq:MeshSolution} is indeed the general solution to \eqref{Eq:ActiveExpanded1}, with the first term being a particular solution and the second term parameterizing the homogeneous solution in the slack variable $y \in \real^c$. Thus, \eqref{Eq:ActivePowerD} is equivalent to
\begin{equation}\label{Eq:ActiveExpanded3}
\bsin(A^{\sf T}\theta) = \psi(v,y)\,.
\end{equation}
{\tb 
It remains only to show that \eqref{Eq:CycleConstraint} and \eqref{Eq:ActiveExpanded3} are equivalent. If $(\theta,V_L) \in \Theta(\frac{\pi}{2}) \times \real^{n}_{>0}$ is a solution of \eqref{Eq:Active}, it follows from \eqref{Eq:ActiveExpanded3} that $\barcsin(\psi(v,y)) = A^{\sf T}\theta + 2\pi k$ for some integer vector $k = (k_1,\ldots,k_{|\mathcal{E}|})^{\sf T} \in \mathbb{Z}^{|\mathcal{E}|}$. Left-multiplying this equality by $C^{\sf T}$ and using that $C^{\sf T}A^{\sf T} = (AC)^{\sf T} = \vzeros[]$, we have
\begin{equation}\label{Eq:ctarcsin}
C^{\sf T}\barcsin(\psi(v,y)) = 2\pi\,C^{\sf T}k\,.
\end{equation}
Since $C \in \{-1,0,1\}^{|\mathcal{E}|\times c}$, $C^{\sf T}k \in \mathbb{Z}^{c}$, and therefore each component of the right-hand side of \eqref{Eq:ctarcsin} is an integer multiple of $2\pi$. It follows that \eqref{Eq:ctarcsin} is equivalent to \eqref{Eq:CycleConstraint}.
%
%
Conversely, if \eqref{Eq:CycleConstraint} holds, then $C^{\sf T}\barcsin(\psi(v,y)) = 2\pi k^\prime$ for some integer vector $k^\prime = (k^\prime_1,\ldots,k^\prime_c)^{\sf T} \in \mathbb{Z}^{c}$. 
Since $C$ has full rank and is totally unimodular, we can always find a $k \in \mathbb{Z}^{|\mathcal{E}|}$ such that $C^{\sf T}k = k^\prime$ \cite[Theorem 5.20]{BK-JV:08}.
The general solution to $C^{\sf T}\barcsin(\psi(v,y)) = 2\pi k^\prime$ may therefore be written as $\barcsin(\psi(v,y)) = 2\pi k + A^{\sf T}\theta$ for some $\theta \in \real^{n +m}$, with $2\pi k$ being the particular solution and $A^{\sf T}\theta$ parameterizing the homogeneous solution. Taking the $\bsin(\cdot)$ of both sides yields \eqref{Eq:ActiveExpanded3}, which completes the proof. 
}
\end{proof}

{\tb The modulo operation in Lemma \ref{Lem:ActiveMesh} is subtle, but is required to capture so-called loop flows; see \cite[Remark 5.3.2]{FD:13} for details.} Equations \eqref{Eq:CycleConstraint} and \eqref{Eq:Candidate} are our desired reformulation of the active power flow equations \eqref{Eq:Active}; \eqref{Eq:MeshSolution} and \eqref{Eq:Candidate} are essentially the explicit solution, while \eqref{Eq:CycleConstraint} enforces that Kichhoff's voltage law holds true around the cycles of the network. The following corollary shows that in radial networks, the formula \eqref{Eq:MeshSolution} which uses the Laplacian stiffness matrix ${\sf L}$ may be replaced by a simpler formula using the branch stiffness matrix $\mathsf{D}$.

\smallskip

\begin{corollary}\label{Cor:AcyclicActive}\textbf{(Active Power Reformulation for Radial Networks):}
If the graph $(\mathcal{N},\mathcal{E})$ describing the network is radial, then the cycle constraints \eqref{Eq:CycleConstraint} are discarded, and the explicit solution $\psi(v,y)$ from Lemma \ref{Lem:ActiveMesh} reduces to
\begin{equation}\label{Eq:AcyclicSolution}
\bsin(A^{\sf T}\theta) = \psi(v) \triangleq [h(v)]^{-1} \mathsf{D}^{-1}p\,,
\end{equation}
where
\begin{equation}\label{Eq:BranchFlows}
p = (p_{\ell\ell},p_{g\ell},p_{gg})^{\sf T} \triangleq (A^{\sf T}A)^{-1}A^{\sf T}P\,
\end{equation}
are the \emph{unique} branch-wise active power flows, satisfying Kirchhoff's current law $P = Ap$. 
\end{corollary}

\begin{proof}.
That the cycle constraints \eqref{Eq:CycleConstraint} can be discarded is clear, as there are no cycles by assumption, so we may set $Cy = \vzeros[|\mathcal{E}|]$ in \eqref{Eq:MeshSolution}. Comparing \eqref{Eq:MeshSolution} and \eqref{Eq:AcyclicSolution} then, we need only prove that $A^{\sf T}{\sf L}^{\dagger}P = \mathsf{D}^{-1}p$. To do so, set $Z \triangleq A^{\sf T}{\sf L}^{\dagger}P$, and note that
$$
A\mathsf{D}Z = A\mathsf{D}A^{\sf T}{\sf L}^{\dagger}P = {\sf L}{\sf L}^{\dagger}P = \Pi P = P\,.
$$
Since $\mathrm{ker}(A) = \emptyset$, we may left-multiply by the left-inverse $(A^{\sf T}A)^{-1}A^{\sf T}$ of $A$ and insert \eqref{Eq:BranchFlows} to obtain
$$
\mathsf{D}Z = p \quad \Leftrightarrow \quad Z = \mathsf{D}^{-1}p\,,
$$
which shows the result. \end{proof}
\smallskip

For later use, we use $h(v)$ from \eqref{Eq:VVectorForm} to expand out the acyclic solution \eqref{Eq:AcyclicSolution} in terms of the branch partitions as
\begin{equation}\label{Eq:ActiveDeveloped}
\begin{pmatrix}
\bsin(\eta_{\ell\ell})\\
\bsin(\eta_{g\ell})\\
\bsin(\eta_{gg})
\end{pmatrix}
=
\begin{pmatrix}
[h_{\ell\ell}(v)]^{-1}\mathsf{D}_{\ell\ell}^{-1}p_{\ell\ell} \\ 
[h_{g\ell}(v)]^{-1}\mathsf{D}_{g\ell}^{-1}p_{g\ell} \\ 
\mathsf{D}_{gg}^{-1}p_{gg} 
\end{pmatrix}\,.
\end{equation}
where
\begin{equation}\label{Eq:eta}
\eta \triangleq (\eta_{\ell\ell},\eta_{g\ell},\eta_{gg})^{\sf T} = A^{\sf T}\theta\,
\end{equation}
are the branch-wise phase angle differences.

\subsection{Reformulation Step 2: Reactive Power Flow}
\label{Sec:Step2}

It is convenient to define an ``unoriented'' version of the incidence matrix $A$, denoted by $|A| \in \real^{(n+m)\times|\mathcal{E}|}$, with all non-zero elements set to $+1$. 
From \eqref{Eq:IncidencePlusMinus}, it is clear that $|A| = A(+)+A(-)$. As we did for the incidence matrix \eqref{Eq:Incidence}, we partition $|A|$ according to bus and branch types as
\begin{align}\label{Eq:IncidenceAbs}
|A| &= \begin{pmatrix}|A|_L \\ |A|_G\end{pmatrix} = 
\begin{pmatrix}
|A|_{L}^{\ell\ell} & |A|_L^{g\ell} & \vzeros[]\\
\vzeros[] & |A|_G^{g\ell} & |A|^{gg}_G
\end{pmatrix}\,.
\end{align}
We now focus on the reactive power flow equations \eqref{Eq:Reactive}. Using Lemma \ref{Lem:EquivFormulas},  \eqref{Eq:Reactive} can be written in vector form as
\begin{equation}\label{Eq:QManip1}
\begin{aligned}
Q_L &= -[V_L][B_{ii}]_{i\in\mathcal{N}_L}V_L\\ &\quad - |A|_L[V_iV_jB_{ij}]_{(i,j)\in\mathcal{E}}\,\bcos(A^{\sf T}\theta)\,.
\end{aligned}
\end{equation}
Adding and subtracting $|A|_L[V_iV_jB_{ij}]_{(i,j)\in\mathcal{E}}\vones[|\mathcal{E}|]$ from the right-hand side, \eqref{Eq:QManip1} becomes
\begin{align*}
Q_L &= -[V_L][B_{ii}]_{i\in\mathcal{N}_L}V_L - |A|_L[V_iV_jB_{ij}]_{(i,j)\in\mathcal{E}}\vones[|\mathcal{E}|]\\& + |A|_L[V_iV_jB_{ij}]_{(i,j)\in\mathcal{E}}(\vones[|\mathcal{E}|]-\bcos(A^{\sf T}\theta))\,.
\end{align*}
Applying Lemma \ref{Lem:EquivFormulas} (i)--(iii) to the first two terms, we obtain
\begin{align}
\begin{aligned}\nonumber
Q_L &= - [V_L]\left(B_{LL} V_L +B_{LG} V_G\right)\\ &\quad + |A|_L[V_iV_jB_{ij}]_{(i,j)\in\mathcal{E}}(\vones[|\mathcal{E}|]-\bcos(A^{\sf T}\theta))
\end{aligned}\\
\begin{aligned}\label{Eq:ReactiveLOL}
&= - [V_L]B_{LL}(V_L-V_L^*)\\ &\quad + |A|_L[V_iV_jB_{ij}]_{(i,j)\in\mathcal{E}}(\vones[|\mathcal{E}|]-\bcos(A^{\sf T}\theta))\,.
\end{aligned}
\end{align}
Working on the first term in \eqref{Eq:ReactiveLOL}, by substituting for $V_L = [V_L^*]v$ from the coordinate change \eqref{Eq:x} and simplifying using the nodal stiffness matrix $\mathsf{S}$ from \eqref{Eq:Qcrit}, we find that
$$
[V_L]B_{LL}(V_L-V_L^*) = 4[v]\mathsf{S}(v-\vones[n])\,.
$$
{\tb For the second term in \eqref{Eq:ReactiveLOL}, again introduce $\eta = A^{\sf T}\theta$ and insert the formula \eqref{Eq:DMatrix5} for $[V_iV_jB_{ij}]_{(i,j)\in\mathcal{E}}$ to obtain
\begin{equation}\label{Eq:ReactiveDeveloped}
\begin{aligned}
Q_L &= -4[v]\mathsf{S}(v-\vones[n])\\ &\quad + |A|_L \mathsf{D}\, [h(v)](\vones[|\mathcal{E}|]-\bcos(\eta))\,,
\end{aligned}
\end{equation}
which is the main result of Step 2. The following lemma summarizes our reformulations.

\smallskip

\begin{lemma}[\bf Reactive Power Reformulation]\label{Lem:ReactivePowerReformulation}
Consider the reactive power flow equation \eqref{Eq:Reactive}. The following two statements are equivalent:
\begin{enumerate}\setlength{\itemsep}{1.5pt}
\item[(i)] $(\theta,V_L) \in \Theta(\pi/2) \times \real^{n}_{>0}$ is a solution of \eqref{Eq:Reactive};
\item[(ii)] $(\eta,v) \in [-\frac{\pi}{2},\frac{\pi}{2}]^{|\mathcal{E}|} \times \real_{>0}^n$ is a solution of \eqref{Eq:ReactiveDeveloped}\,.
\end{enumerate}
\end{lemma}
}

\smallskip

\subsection{Reformulation Step 3: Eliminate the Phase Angles}
\label{Sec:Step3}

{\tb 
We now combine the two reformulations \eqref{Eq:Candidate} and \eqref{Eq:ReactiveDeveloped}.
To begin, we apply $\sin^2 \eta_{ij} + \cos^2 \eta_{ij} = 1$ component-wise to \eqref{Eq:Candidate} and to obtain
\begin{equation}\label{Eq:cos-meshed}
\bcos(\eta) = \sqrt{\vones[|\mathcal{E}|]-[\psi(v,y)]\psi(v,y)}\,,
\end{equation}
where we have selected the positive square root for all components, as we are interested in phase angle vectors $\theta \in \Theta(\pi/2)$. Next, left-multiply both sides of \eqref{Eq:ReactiveDeveloped} by $-\frac{1}{4}\mathsf{S}^{-1}[v]^{-1}$ and rearrange to obtain
\begin{equation}\label{Eq:ReactiveDeveloped2}
\begin{aligned}
v &= \vones[n] - \frac{1}{4}\mathsf{S}^{-1}[Q_L][v]^{-1}\vones[n]\\
&\quad +\frac{1}{4}\mathsf{S}^{-1}[v]^{-1}|A|_L\mathsf{D}\,[h(v)]\,(\vones[|\mathcal{E}|]-\bcos(\eta))\,.
\end{aligned}
\end{equation}
Substituting \eqref{Eq:cos-meshed} into \eqref{Eq:ReactiveDeveloped2}, we may combine Lemma \ref{Lem:ActiveMesh} and Lemma \ref{Lem:ReactivePowerReformulation} to obtain our most general main modeling result.
}

\smallskip

\begin{theorem}\label{Thm:FPPFMeshed}\textbf{(Fixed-Point Power Flow for Meshed Networks):}
Consider the coupled power flow equations \eqref{Eq:Active}--\eqref{Eq:Reactive}. The following two statements are equivalent:
\begin{enumerate}\setlength{\itemsep}{1.5pt}
\item[(i)] $(\theta,V_L) \in \Theta(\pi/2) \times \real^{n}_{>0}$ is a solution of \eqref{Eq:Active}--\eqref{Eq:Reactive};
\item[(ii)] $(v,y) \in \real_{>0}^n \times \real^{c}$ is a solution of the fixed-point power flow
\begin{subequations}\label{Eq:FPPFMeshed}
\begin{align}\label{Eq:FPPFMeshed1}
\begin{aligned}
v &= f_{\rm mesh}(v,y) \triangleq \vones[n] - \frac{1}{4}\mathsf{S}^{-1}[Q_L][v]^{-1}\vones[n]\\
&\quad +\frac{1}{4}\mathsf{S}^{-1}[v]^{-1}|A|_L\mathsf{D}\,[h(v)]\,u(v,y)\,,\\
\end{aligned}\\
\begin{aligned}\label{Eq:FPPFMeshed2}
\vzeros[c] &= C^{\sf T}\barcsin(\psi(v,y))\, \qquad {\tb (\boldsymbol{\mathrm{mod}}\,2\pi)}\,.
\end{aligned}
\end{align}
\end{subequations}
where
\begin{subequations}
\begin{align}\label{Eq:uxy}
u(v,y) &\triangleq \vones[|\mathcal{E}|] - \sqrt{\vones[|\mathcal{E}|]-[\psi(v,y)]\psi(v,y)}\,\\
\label{Eq:psixy}
\psi(v,y) &= [h(v)]^{-1}\left(A^{\sf T}{\sf L}^{\dagger}P + \mathsf{D}^{-1}Cy\right)\,,
\end{align}
\end{subequations}
where $h(v)$ is as in \eqref{Eq:VVectorForm}, with the angular differences $\eta = A^{\sf T}\theta\,\, \tb (\boldsymbol{\mathrm{mod}}\,2\pi) \in [-\frac{\pi}{2},\frac{\pi}{2}]^{|\mathcal{E}|}$ recovered via $\eta = \barcsin(\psi(v,y))$\,.
\end{enumerate}
\end{theorem}

\smallskip
\smallskip

The model \eqref{Eq:FPPFMeshed1}--\eqref{Eq:FPPFMeshed2} depends only on the scaled \PQ bus voltage magnitudes $v \in \real^n_{>0}$ and on the slack variables $y \in \real^c$ which enforce the cycle constraints arising from Kirchhoff's voltage law. Phase angles are completely absent, and are recovered as ``outputs'' by applying $\arcsin(\cdot)$ component-wise to \eqref{Eq:psixy}. The terminology fixed-point power flow comes from the special form that the model takes for radial networks.

\smallskip

\begin{corollary}\label{Cor:FPPFAcyclic}\textbf{(Fixed-Point Power Flow for Radial Networks I):}
Consider the coupled power flow equations \eqref{Eq:Active}--\eqref{Eq:Reactive} and assume that the graph $(\mathcal{N},\mathcal{E})$ describing the network is radial. The following two statements are equivalent:
\begin{enumerate}\setlength{\itemsep}{1.5pt}
\item[(i)] $(\theta,V_L) \in \Theta(\pi/2) \times \real^{n}_{>0}$ is a solution of \eqref{Eq:Active}--\eqref{Eq:Reactive};
\item[(ii)] $v \in \real^n_{>0}$ is a fixed point\footnote{A fixed point of $f_{\rm radial}$ is a point $v \in \real^n_{>0}$ satisfying $f_{\rm radial}(v) = v$.} of the mapping $\map{f_{\rm radial}}{\real^n_{>0}}{\real^n}$ defined by
\begin{equation}\label{Eq:f}
\begin{aligned}
f_{\rm radial}(v) &\triangleq \vones[n]-\frac{1}{4}\mathsf{S}^{-1}[Q_L][v]^{-1}\vones[n]\\ &\quad + \frac{1}{4}\mathsf{S}^{-1}[v]^{-1}|A|_L\mathsf{D}\,[h(v)]\,u(v)\,,
\end{aligned}
\end{equation}
where
\begin{subequations}
\begin{align}\label{Eq:ux}
u(v) &\triangleq \vones[|\mathcal{E}|] - \sqrt{\vones[|\mathcal{E}|]-[\psi(v)]\psi(v)}\,\\
\label{Eq:psix}
\psi(v) &= [h(v)]^{-1}\mathsf{D}^{-1}p\,,
\end{align}
\end{subequations}
with $h(v)$ as in \eqref{Eq:VVectorForm} and the angular differences $\eta = A^{\sf T}\theta \in [-\frac{\pi}{2},\frac{\pi}{2}]^{|\mathcal{E}|}$ recovered via $\eta = \barcsin(\psi(v))$\,.
\end{enumerate}
\end{corollary}

\begin{proof}.
By Corollary \ref{Cor:AcyclicActive}, $\psi(v,y)$ from \eqref{Eq:psixy} reduces to $\psi(v)$ in \eqref{Eq:psix}, and hence $u(v,y)$ from \eqref{Eq:uxy} reduces to $u(v)$ in \eqref{Eq:ux}. Again by Corollary \ref{Cor:AcyclicActive}, we may omit the cycle constraint \eqref{Eq:FPPFMeshed2}. It follows that $f_{\rm mesh}(v,y)$ as defined in \eqref{Eq:FPPFMeshed1} reduces to $f_{\rm radial}(v)$ as in \eqref{Eq:f}.
\end{proof}

\smallskip

The function $f_{\rm radial}(v)$ may be written in an alternative form which emphasizes the importance the branch partitioning $\mathcal{E} = \mathcal{E}^{\ell\ell} \cup \mathcal{E}^{g\ell} \cup \mathcal{E}^{gg}$. This form is less compact, but more explicit, and will be used for analysis purposes in Part II.

\smallskip

\begin{corollary}\label{Cor:FPPFAcyclicAlt}\textbf{(Fixed-Point Power Flow for Radial Networks II):}
With the notation $u(v) = (u_{\ell\ell}(v),u_{g\ell}(v),u_{gg}(v))^{\sf T}$, an equivalent expression for the fixed point map $f_{\rm radial}(v)$ in Corollary \ref{Cor:FPPFAcyclic} is
\begin{equation}\label{Eq:FixedPoint}
\begin{aligned}
f_{\rm radial}(v) &= \vones[n]-\frac{1}{4}\mathsf{S}^{-1}[Q_L][v]^{-1}\vones[n]\\
&+ \frac{1}{4}\mathsf{S}^{-1}|A|_L^{g\ell}\mathsf{D}_{g\ell}u_{g\ell}(v)\\
&+ \frac{1}{4}\mathsf{S}^{-1}A_L^{\ell\ell}(+)\,[A_L^{\ell\ell}(-)^{\sf T}v]\mathsf{D}_{\ell\ell}u_{\ell\ell}(v)\\
&+ \frac{1}{4}\mathsf{S}^{-1}A_L^{\ell\ell}(-)\,[A_L^{\ell\ell}(+)^{\sf T}v]\mathsf{D}_{\ell\ell}u_{\ell\ell}(v)\,.
\end{aligned}
\end{equation}
\end{corollary}

\begin{proof}.
See appendix.
\end{proof}

\subsection{Discussion of FPPF Derivation}
\label{Sec:Discussion}

The fixed-point power flow model \eqref{Eq:FPPFMeshed} is parameterized by the stiffness matrices $\mathsf{D}, \mathsf{L}$, and $\mathsf{S}$ and the power demands $P$ and $Q_L$. As discussed in Section \ref{Sec:Stiffness}, the stiffness matrices encode both the topology of the network and the values of relevant parameters such as line susceptances and generator voltages. A particularly important observation is that the active and reactive power loads $P$ and $Q_L$ each enter the model multiplied by the inverse of a stiffness matrix ($A^{\sf T}{\sf L}^{\dagger}P$ in the first case, $\mathsf{S}^{-1}[Q_L]$ in the second case). The FPPF model reveals that these very specific combinations of network parameters and loads are the important quantities to focus on.
\smallskip

In the radial case, the map $f_{\rm radial}$ in \eqref{Eq:FixedPoint} consists of five distinct, easily interpretable terms:
\begin{itemize}
\item the first term is the open-circuit voltage profile of the network in the scaled variables $v_i = V_i/V_i^*$;
\item the second term is a reformulation of the decoupled reactive power flow equation; see \cite{jwsp-fd-fb:14c}.
\item the third term proportional to $u_{g\ell}(v)$ accounts for the influence of active power branch flows between \PV and \PQ buses on the voltage magnitudes at \PQ buses.
\item the fourth and fifth terms proportional to $u_{\ell\ell}(v)$ account for the influence of active power branch flows between \PQ buses on \PQ bus voltage magnitudes. There are two terms to account for this because a branch-wise flow of active power between two \PQ buses affects the voltage magnitude at both ends of the branch.
\end{itemize}

A more detailed examination of the internal structure of these terms is deferred to Part II \cite{JWSP:17b}. Conceptually, the FPPF model is similar Baran-Wu branch flow model \cite{SHL:14a} in that phase angles are absent. In contrast to the branch flow model, where sending-end branch power flows and currents are state variables, the voltages $v$ are the only state variables in \eqref{Eq:FixedPoint}; this difference in complexity can perhaps be attributed to the lossless character of the network considered here.

{
\subsection{Approximate Solution to Lossless Power Flow}
\label{Sec:Linearization}

We now leverage the FPPF model to derive an explicit approximate solution to the power flow equations \eqref{Eq:Active}--\eqref{Eq:Reactive}. By construction, the FPPF model \eqref{Eq:FPPFMeshed} has the property that under open-circuit conditions (when $Q_L = \vzeros[n]$ and $P = \vzeros[n+m]$), it holds that $(v,y,\eta) = (\vones[n],\vzeros[c],\vzeros[|\mathcal{E}|])$ is a solution:
$$
f_{\rm mesh}(\vones[n],\vzeros[c]) = \vones[n] \quad \text{and} \quad \psi(\vones[n],\vzeros[c]) = \vzeros[|\mathcal{E}|]\,.
$$
This is the solution we desire in practice: a high-voltage solution ($v \approx \vones[n]$) {\tb with small angular differences ($\eta \approx \vzeros[|\mathcal{E}|]$)}. By Taylor expanding each side of \eqref{Eq:FPPFMeshed} around this solution, we can obtain an explicit expression for the ``linearized'' power flow solution $(\eta_{\rm lin},v_{\rm lin})$ to lowest order in both $Q_L$ and $P$. Expanding \eqref{Eq:FPPFMeshed2} around the open-circuit solution, to first order we find that
\begin{align*}
\vzeros[c] &= C^{\sf T}(A^{\sf T}{\sf L}^{\dagger}P + {\sf D}^{-1}Cy_{\rm lin})\,.
\end{align*}
By construction, $C$ is full column rank with $\mathrm{Im}(C) = \mathrm{ker}(A)$ (Section \ref{Sec:Step1}). It follows that $AC = \vzeros[]$ and that $C^{\sf T}{\sf D}^{-1}C$ is invertible. We therefore conclude that $y_{\rm lin} = \vzeros[c]$ is the approximate solution for the cyclic slack variables. While we omit the details, expanding both \eqref{Eq:FPPFMeshed1} and the expression $\eta = \barcsin(\psi)$ to lowest order in $P$ and $Q_L$ yields
\begin{subequations}\label{Eq:LinearizedFPPF}
\begin{align}
\label{Eq:LinearizedFPPF1}
v_{\rm lin} &= \vones[n] - \frac{1}{4}{\sf S}^{-1}Q_L + \frac{1}{8}{\sf S}^{-1}|A|_L {\sf D}[A^{\sf T}{\sf L}^{\dagger}P]A^{\sf T}{\sf L}^{\dagger}P\\
\label{Eq:LinearizedFPPF2}
\eta_{\rm lin} &= A^{\sf T}\theta_{\rm lin} = A^{\sf T}{\sf L}^{\dagger}P\,.
\end{align}
\end{subequations}
The first two terms in \eqref{Eq:LinearizedFPPF1} are the approximate solution for decoupled reactive power flow, as obtained in \cite{BG-JWSP-FD-SZ-FB:13zb,jwsp-fd-fb:14c}. The third term is novel, and captures how \textemdash{} to lowest order \textemdash{} active power injections affect voltage magnitudes quadratically. The linearization \eqref{Eq:LinearizedFPPF2} is the standard DC Power Flow approximation \cite{BS-JJ-OA:09}. Taken together, \eqref{Eq:LinearizedFPPF} is an approximate solution to the lossless power flow equations. Intuitively, this approximation will be accurate when (i) loop flows of power are insignificant, and (ii) the quantities $\|{\sf S}^{-1}Q_L\|$ and $\|A^{\sf T}{\sf L}^{\dagger}P\|$ are both small.\footnote{{\tb In Part II we will show for radial networks that appropriately restricting these two quantities guarantees the existence of a power flow solution.}} We note that higher-order explicit approximations can be obtained by retaining higher-order terms in the Taylor expansion.

}

\section{Numerical Experiments}
\label{Sec:Simulations}

We now present simulation studies on standard test cases to test the computational performance of the FPPF model as well as the accuracy of the approximate solution \eqref{Eq:LinearizedFPPF}. Algorithm \ref{Alg:FPPF} describes the implementation of the FPPF used for the tests described here; many variations and computational refinements are of course possible. The basic approach is to iterate the mapping $v_{k+1} = f_{\rm mesh}(v_k,y_k)$ to update the scaled voltages $v$. For meshed networks the slack variables $y$ must also be updated. {\tb While many options are possible, we will use a Newton step}
{\tb
\begin{equation}\label{Eq:NewtonCycle}
J_{\mathrm{cycle},k}(y_{k+1}-y_k) = -C^{\sf T}\barcsin(\psi(v_k,y_k))\,,
\end{equation}
}
where
\begin{equation}\label{Eq:JacCycle}
J_{\mathrm{cycle},k} \triangleq C^{\sf T}\left(I_{|\mathcal{E}|}-[\psi_k]^2\right)^{-\frac{1}{2}}[h(v_k)]^{-1}\mathsf{D}^{-1}C
\end{equation}
is the Jacobian matrix of \eqref{Eq:FPPFMeshed2}.
After a desired relative tolerance $\epsilon$ is reached, the algorithm terminates.

\begin{algorithm}

\SetKwInput{KwData}{Inputs}
\SetKwInput{KwResult}{Outputs}

 \KwData{Power flow data, Iteration Tolerance $\epsilon$}
 \KwResult{Power flow solution $(A^{\sf T}\theta,V_L)$, cycle variable $y$}
 Construct $V_L^*, \mathsf{D},\mathsf{S},{\sf L}, C$\\
 Initialization: $v = \vones[n], y = \vzeros[c], \psi =\vzeros[|\mathcal{E}|]$\\
 \While{$\max\{\|v_{k}-v_{k-1}\|_{\infty}\,,\,\|\psi_k-\psi_{k-1}\|_{\infty}\} > \epsilon$}{
    $v \leftarrow f_{\rm mesh}(v,y)$\\
   \If{$c > 0$}{
	$y \leftarrow$ Newton Step on $C^{\sf T}\barcsin(\psi(v,y)) = \vzeros[c]$\\	
   }
   $\psi \leftarrow [h(v)]^{-1}\left(A^{\sf T}{\sf L}^{\dagger}P + \mathsf{D}^{-1}Cy\right)$\\
  }
  \KwRet{$V_L \leftarrow [V_L^*]\,v$, $A^{\sf T}\theta \leftarrow \barcsin(\psi)$, $y$}
 \caption{Fixed-Point Power Flow Iteration}
 \label{Alg:FPPF}
\end{algorithm}

{\tb 
\begin{remark}\label{Rem:Computational}\textbf{(Computational Remarks):}
Various combinations of constants such as $A^{\sf T}{\sf L}^{\dagger}P$, ${\sf D}^{-1}C$, and $|A|_L{\sf D}$ appear in the iterations of Algorithm \ref{Alg:FPPF}; these can be precomputed then stored for future use. For example, one may solve the sparse linear equation ${\sf L}z = P$, calculate $A^{\sf T}z$, then store the result. The iteration $v \leftarrow f_{\rm mesh}(v,y)$ should be computed by right-multiplying \eqref{Eq:FPPFMeshed1} by ${\sf S}$ and solving the corresponding sparse linear system. A factorization (e.g., Cholesky) of ${\sf S}$ can be stored, and each update will then require only one forward-backward substitution; other sparsity-exploiting techniques could also be applied. For the Newton step \eqref{Eq:NewtonCycle}, the Jacobian \eqref{Eq:JacCycle} is sparse and the variable portion $\left(I_{|\mathcal{E}|}-[\psi_k]^2\right)^{-\frac{1}{2}}[h(v_k)]^{-1}$ is diagonal. In summary, each iteration of Algorithm \ref{Alg:FPPF} requires the solution of an $n \times n$ system of sparse equations (with a constant coefficient matrix) and the solution of a $c \times c$ system of sparse equations for the slack Newton step. \hfill \oprocend
\end{remark}
}

{\tb
\begin{remark}\label{Rem:Compare}\textbf{(Dimensional Comparison To Standard Power Flow Methods):}
Classic Newton-Raphson or Fast Decoupled Load Flow implementations iterate on the $n+m$ phase angles and $n$ voltage magnitudes. In contrast, the Algorithm \ref{Alg:FPPF} iterates on the $n$ voltage magnitudes and $c$ slack variables. Typically $c < n+m$, and therefore the FPPF algorithm requires the solution of (sometimes substantially) smaller systems of linear equations than these implementations. For example, the 39 bus New England system has $c = 8$, the 2383 Polish system has $c = 514$, and the 9241 PEGASE system has $c = 6809$. We also note that many different solution techniques (including Newton's method) could be applied to the system of $n+c$ nonlinear equations \eqref{Eq:FPPFMeshed}. Algorithm \ref{Alg:FPPF} is just one option, but is well-motivated by our development in Part II. \hfill \oprocend
\end{remark}
}

We apply Algorithm \ref{Alg:FPPF} and the approximate solution \eqref{Eq:LinearizedFPPF} to the standard MATPOWER test cases \cite{RDZ-CEM-DG:11,CZ-SF-JM-PP:16}. In all cases, branch and shunt conductances were set to zero, in line with our main theoretical assumption.\footnote{{\tb For context on this assumption, the mean branch $R/X$ ratios of the networks in Table \ref{Tab:Iterations1} are approximately 0.3, 0.2, 0.4, 0.15, 0.4, 0.1, 0.3, 0.2, 0.1, 0.3, 0.2, and 0.2. The mean differences between the voltage solutions computed with and without resistances are 0.007, 0.006, 0.010, 0.004, 0.024, 0.006, 0.003, 0.010, 0.005, 0.005, 0.005, and 0.007 per unit, respectively.}} The cycle-space matrix $C \in \{-1,0,1\}^{|\mathcal{E}|\times c}$ was generated using $\texttt{C = null(A,"r")}$ in MATLAB. Simulation results are shown in Table \ref{Tab:Iterations1}. The second column shows the number of iterations required by Algorithm \ref{Alg:FPPF} to reach a relative solution tolerance of $\epsilon = 0.001$ {\tb p.u. on the voltage magnitudes}. As can be seen, the iterations converge to the high-voltage solution quickly. The number of iterations is essentially independent of system size, and in all 
cases the solution agrees with the one determined by MATPOWER. Columns 3 and 4 show the max and mean prediction errors
$$
\delta_{\rm max} = \|V_L - [V_L^*]v_{\rm lin}\|_{\infty}\,,\quad \delta_{\rm avg} = \frac{1}{n}\|V_L - [V_L^*]v_{\rm lin}\|_{1}\,
$$ 
between the exact voltage magnitude solution $V_L$ and the approximate solution determined by \eqref{Eq:LinearizedFPPF}, in per unit. In most of the cases the maximum error is quite small, with larger test cases showing larger maximum errors. Across all cases however, the mean error is consistently quite small, indicating a good overall approximation.

\begin{table}[h!]
\begin{center}
\caption{Testing of Fixed-Point Power Flow}
{\renewcommand{\arraystretch}{1}
\begin{tabular}{@{\extracolsep{4pt}}|l|c|c|c|c|c|@{}}
\toprule
& \multicolumn{3}{c|}{Base Load} & \multicolumn{2}{c|}{High Load}\\
\midrule
\multirow{2}{*}{Test Case} & FPPF & $\delta_{\rm max}$ & $\delta_{\rm avg}$ & FPPF & $\delta_{\rm max}$\\
&  Iters. & (p.u.) & (p.u.) &  Iters. & (p.u.)\\
\midrule
14 bus system 			& 4		&	0.001 	&	0.000 	& 	8	&	0.090\\
RTS 24 					& 4		&	0.003 	&	0.001 	&	8	&	0.081\\
30 bus system 			& 4		&	0.003 	&	0.002 	& 	8	&	0.104\\
New England 39 			& 4		&	0.006 	&	0.004 	& 	8	&	0.086\\
57 bus system 			& 5		&	0.011 	&	0.003 	& 	8	&	0.118\\
RTS '96 (3 area)		& 4		&	0.003 	&	0.001 	& 	8	&	0.084\\
118 bus system			& 3		&	0.001 	&	0.000 	& 	7	&	0.054\\
300 bus system 			& 6		&	0.022 	&	0.004 	& 	8	&	0.059\\
PEGASE 1,354 			& 5		&	0.011	&	0.001 	&   8	&	0.070\\	
Polish 2,383 wp		 	& 4		& 	0.003	&	0.000 	& 	8	&	0.078\\
PEGASE 2,869 			& 5		&	0.015	&	0.002 	&   8	&	0.098\\	
PEGASE 9,241 			& 6		&	0.063	&	0.003 	&   9	&	0.133\\	
\bottomrule
\end{tabular}
}
\label{Tab:Iterations1}
\end{center}
\end{table}

To examine performance in more heavily loaded networks, each case was loaded along the base case direction 90\% of the way to the power flow insolvability boundary, as determined by continuation power flow \texttt{cpf} in {MATPOWER}. The previous experiments were repeated, and the results are shown in columns four and five of Table \ref{Tab:Iterations1}. As would be expected, convergence to the power flow solution now takes more iterations, and the approximate solution \eqref{Eq:LinearizedFPPF} yields less accurate predictions; mean error $\delta_{\rm avg}$ is smaller, but omitted.

For base case loading, the voltage profile results for the 39 bus system are plotted in Figure \ref{Fig:NE39}. As stated, the profile obtained by iterating the FPPF (blue crosses) coincides with the solution determined by MATPOWER (black circles). The approximate solution \eqref{Eq:LinearizedFPPF} is plotted in red. The approximate solution is quite accurate, but systematically overestimates the voltage values; this behaviour is consistent across all test cases.

\begin{figure}[ht!]
\begin{center}
\includegraphics[width=0.9\columnwidth]{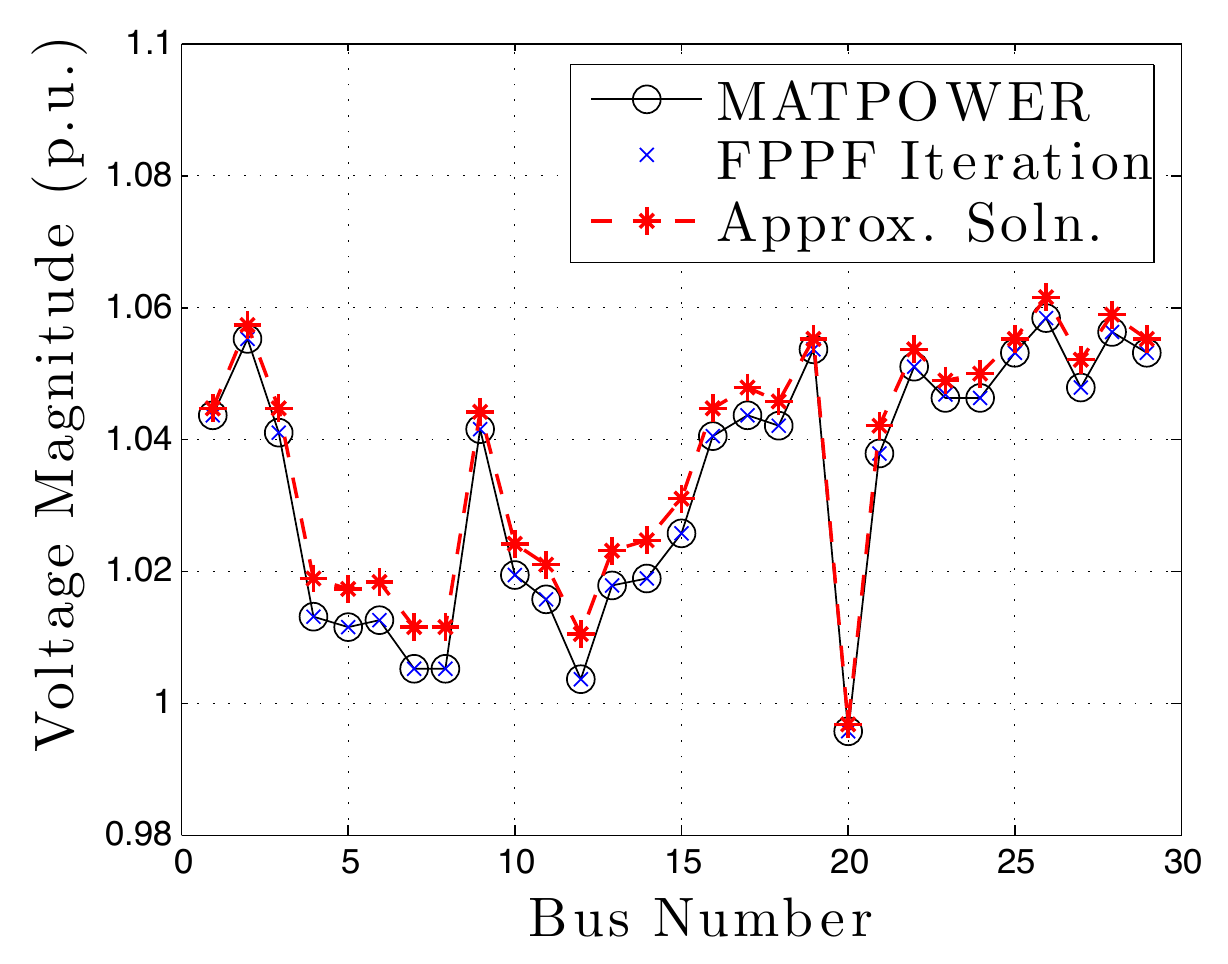}
\caption{Calculated voltage profiles for 39 bus system.}
\label{Fig:NE39}
\end{center}
\end{figure}

{\tb
For the base case IEEE 300 bus test system, Figure \ref{Fig:Compare1} compares the convergence of the FPPF to the stock implementations of Newton-Raphson (NR) and Fast-Decoupled Load Flow (FDLF) from MATPOWER. All three solvers were initialized from a flat start; the vertical axis is the $2$-norm of the difference between current iterate of voltage magnitudes and the exact voltage magnitude solution vector. Both the FPPF and the FDLF show linear convergence (with slightly different rates) while the NR iterates converge quadratically. Figure \ref{Fig:Compare2} repeats the comparison for heavy loading. In this case the NR requires only one additional iteration to reach machine precision, while the FPPF and FDLF iteration counts double; this is consistent across all cases. We conclude that the FPPF algorithm convergence is similar or slightly favourable compared to FDLF, but does not approach that of NR.

\begin{figure}[ht!]
\begin{center}
\includegraphics[width=0.9\columnwidth]{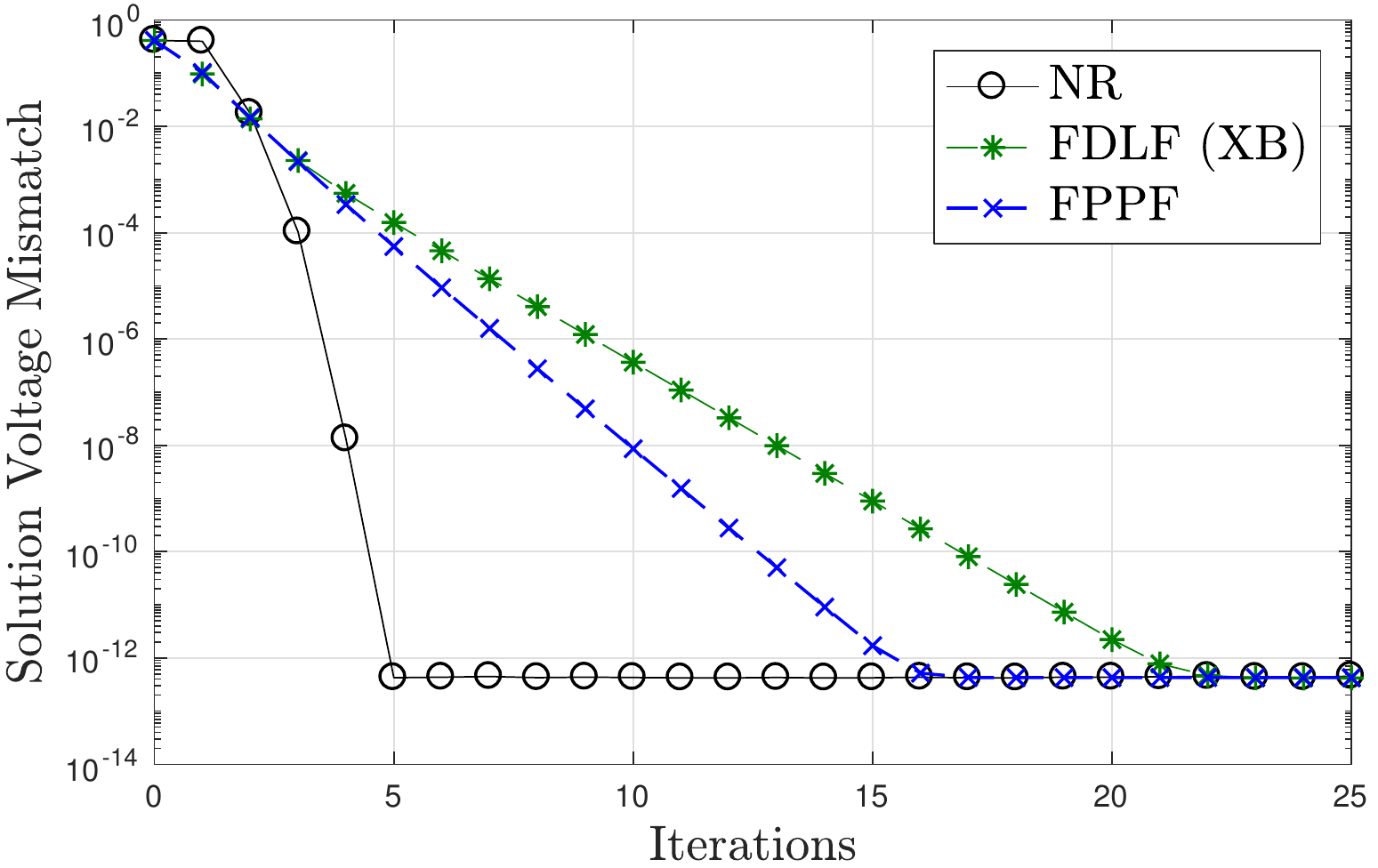}
\caption{Convergence comparison for NR, FDLF, and FPPF solvers for base case loading of the IEEE 300 bus system.}
\label{Fig:Compare1}
\end{center}
\end{figure}

\begin{figure}[ht!]
\begin{center}
\includegraphics[width=0.9\columnwidth]{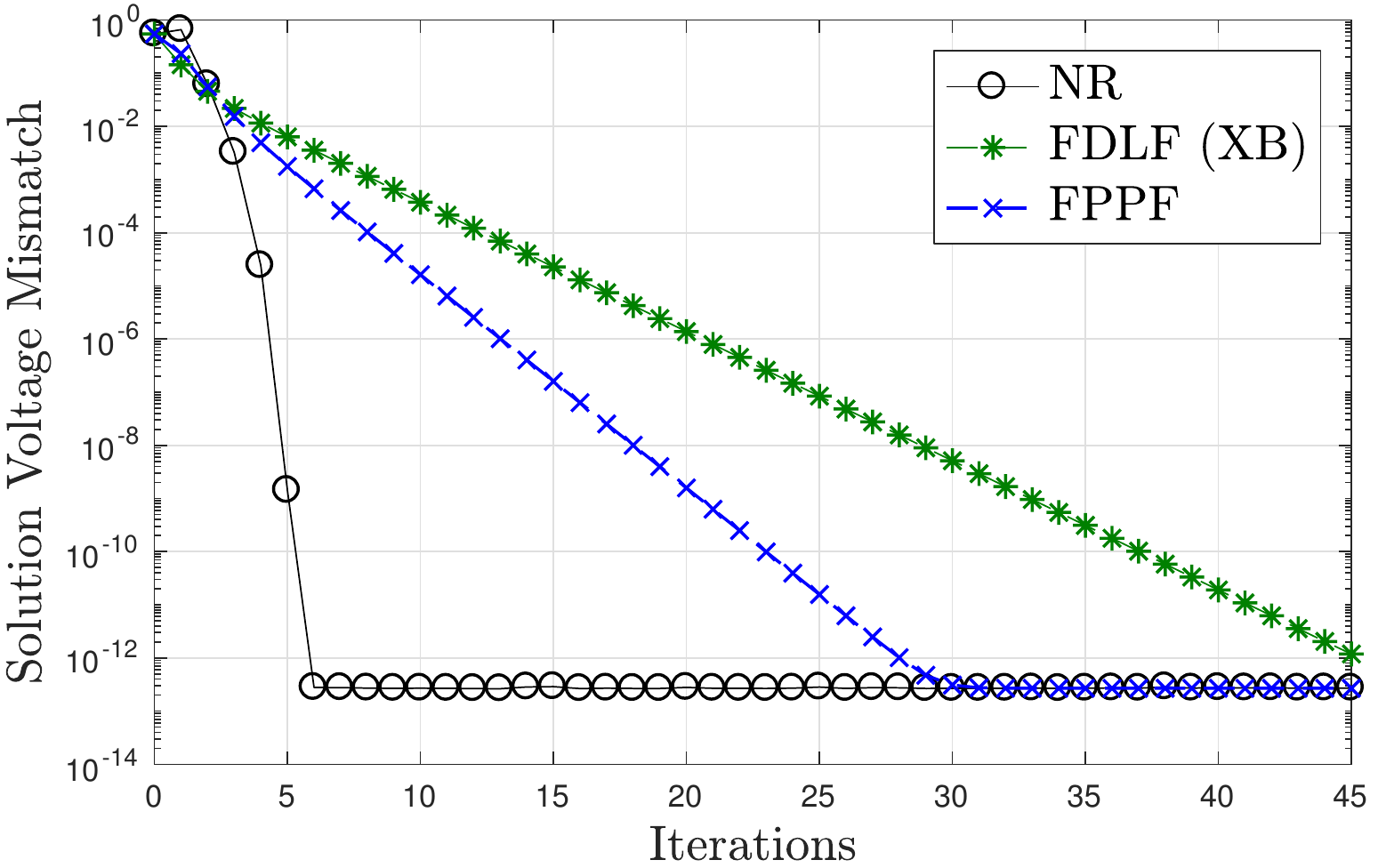}
\caption{Convergence comparison for NR, FDLF, and FPPF solvers for heavy loading of the IEEE 300 bus system.}
\label{Fig:Compare2}
\end{center}
\end{figure}
}

{\tb The linear convergence of the FPPF is consistent with our results in Part II, where we will show \textemdash{} for a subclass of radial networks \textemdash{} that the FPPF is a contraction mapping whenever there exists a high-voltage solution. Those theoretical results do not generalize to the cases considered numerically here. However, they do suggest that contractivity of the FPPF and the existence of a high-voltage solution will go hand-in-hand.
}

{\tb As a final test, we examine the sensitivity of Algorithm \ref{Alg:FPPF} to initialization. If the FPPF is a contraction mapping on a large subset of voltage-space, we would expect the convergence of Algorithm \ref{Alg:FPPF} to be insensitive to the choice of initialization. We consider the IEEE 118 bus test system at base case loading. We fix the angle initial condition at $\theta = \vzeros[n+m]$, and generate 1000 random voltage magnitude initial conditions, each with components pulled uniformly from the interval $[1-\alpha,1+\alpha]$ for various values of $\alpha$. For each initial condition, we run the NR, FDLF, and FPPF algorithms. If the iterates converge to the known high-voltage solution, we mark the test successful; otherwise, we say the solver has failed. Table \ref{Tab:Initial} shows the fraction of successful tests for various values of $\alpha$.
}

\begin{table}[h!]
\begin{center}
\caption{Solver success rates under random initializations for the IEEE 118 bus system.}
{\renewcommand{\arraystretch}{1}
\begin{tabular}{|l|c|c|c|}
\toprule
IC Spread ($\alpha$) & NR & FDLF & FPPF\\
\midrule
0.05 & 0.98 & 0.98 & 1.00\\
0.10 & 0.53 & 0.53 & 1.00\\
0.15 & 0.18 & 0.18 & 1.00\\
0.2 & 0.03 & 0.03 & 1.00\\
0.3 & 0.00 & 0.00 & 1.00\\
0.5 & 0.00 & 0.00 & 1.00\\
0.7 & 0.00 & 0.00 & 0.99\\
0.9 & 0.00 & 0.00 & 0.99\\
\bottomrule
\end{tabular}
}
\label{Tab:Initial}
\end{center}
\end{table}

{\tb For very small values of $\alpha$ (i.e., initial conditions close to a flat start) all three solvers behave similarly. As $\alpha$ increases, both the NR and FDLF solvers increasingly struggle to calculate the high-voltage solution; for some initializations these solvers converge to a low-voltage solution, for other initializations they diverge. On the larger test systems, the NR and FDLF are even more sensitive than suggested by Table \ref{Tab:Initial}. The results show that for NR and FDLF, even one mildly perturbed component of an otherwise tolerable initial condition causes failure to converge to the high-voltage solution,
In contrast, the FPPF iteration recovers the high-voltage solution from nearly every constructed initial condition.}

\section{Conclusions}
\label{Sec:Conclusions}

We have developed a new formulation of lossless power flow equations, which we term the fixed-point power flow. The model is naturally parameterized in terms of the power network stiffness matrices, which concisely encode the topology and parameters of the network, and leads immediately to an approximate solution of the power flow equations. We then proposed an algorithm for solving the power flow equations based on the FPPF. {\tb Numerical testing shows that this algorithm converges quickly for base case and stressed conditions. The convergence is linear, at a rate comparable or favorable compared to fast-decoupled methods, and is very insensitive to initialization.}

In Part II \cite{JWSP:17b} of this paper, we restrict ourselves to the case of radial networks, and leverage the fixed-point power flow to derive sufficient and tight conditions for the existence and uniqueness of a stable power flow solution. The analysis presented there will also provide guarantees for the convergence of the FPPF iteration $v_{k+1} = f_{\rm radial}(v_k)$ for a subclass of radial networks. Future directions and open problems are deferred to the conclusions of Part II.


%
\IEEEpeerreviewmaketitle

\ifCLASSOPTIONcaptionsoff
  \newpage
\fi


\bibliographystyle{IEEEtran}
\bibliography{alias,Main,JWSP,New}

\begin{thebibliography}{10}
\providecommand{\url}[1]{#1}
\csname url@samestyle\endcsname
\providecommand{\newblock}{\relax}
\providecommand{\bibinfo}[2]{#2}
\providecommand{\BIBentrySTDinterwordspacing}{\spaceskip=0pt\relax}
\providecommand{\BIBentryALTinterwordstretchfactor}{4}
\providecommand{\BIBentryALTinterwordspacing}{\spaceskip=\fontdimen2\font plus
\BIBentryALTinterwordstretchfactor\fontdimen3\font minus
  \fontdimen4\font\relax}
\providecommand{\BIBforeignlanguage}[2]{{%
\expandafter\ifx\csname l@#1\endcsname\relax
\typeout{** WARNING: IEEEtran.bst: No hyphenation pattern has been}%
\typeout{** loaded for the language `#1'. Using the pattern for}%
\typeout{** the default language instead.}%
\else
\language=\csname l@#1\endcsname
\fi
#2}}
\providecommand{\BIBdecl}{\relax}
\BIBdecl

\bibitem{JM-JWB-JRB:08}
J.~Machowski, J.~W. Bialek, and J.~R. Bumby, \emph{Power System Dynamics},
  2nd~ed.\hskip 1em plus 0.5em minus 0.4em\relax John Wiley \& Sons, 2008.

\bibitem{DKM-DM-MN:16}
D.~K. Molzahn, D.~Mehta, and M.~Niemerg, ``Towards topologically-based upper
  bounds on the number of power flow solutions,'' in \emph{{A}merican {C}ontrol
  {C}onference}, Boston, MA, USA, Jul. 2016, pp. 5927--5932.

\bibitem{PWS-MAP:90}
P.~W. Sauer and M.~A. Pai, ``Power system steady-state stability and the
  load-flow {J}acobian,'' \emph{IEEE Transactions on Power Systems}, vol.~5,
  no.~4, pp. 1374--1383, 1990.

\bibitem{FD-MC-FB:11v-pnas}
F.~D{\"o}rfler, M.~Chertkov, and F.~Bullo, ``Synchronization in complex
  oscillator networks and smart grids,'' \emph{Proceedings of the National
  Academy of Sciences}, vol. 110, no.~6, pp. 2005--2010, 2013.

\bibitem{SHL:14a}
S.~H. Low, ``Convex relaxation of optimal power flow, part i: Formulations and
  equivalence,'' \emph{IEEE Transactions on Control of Network Systems},
  vol.~1, no.~1, pp. 15--27, 2014.

\bibitem{SHL:14b}
------, ``Convex relaxation of optimal power flow, part ii: Exactness,''
  \emph{IEEE Transactions on Control of Network Systems}, vol.~1, no.~2, pp.
  1--13, 2014.

\bibitem{JST-SAN:89}
J.~S. Thorp and S.~A. Naqavi, ``Load flow fractals,'' in \emph{{IEEE} Conf.\ on
  Decision and Control}, Tampa, FL, USA, Dec 1989, pp. 1822--1827 vol.2.

\bibitem{JST-SAN-HDC:90}
J.~S. Thorp, S.~A. Naqavi, and H.~D. Chiang, ``More load flow fractals,'' in
  \emph{{IEEE} Conf.\ on Decision and Control}, Honolulu, HI, USA, Dec 1990,
  pp. 3028--3030.

\bibitem{DM-DKM-KT:16}
D.~Mehta, D.~K. Molzahn, and K.~Turitsyn, ``Recent advances in computational
  methods for the power flow equations,'' in \emph{{A}merican {C}ontrol
  {C}onference}, Boston, MA, USA, Jul. 2016, pp. 1753--1765.

\bibitem{TVC-CV:98}
T.~Van~Cutsem and C.~Vournas, \emph{Voltage Stability of Electric Power
  Systems}.\hskip 1em plus 0.5em minus 0.4em\relax Springer, 1998.

\bibitem{MEB-FFW:89}
M.~E. Baran and F.~F. Wu, ``Optimal capacitor placement on radial distribution
  systems,'' \emph{IEEE Transactions on Power Delivery}, vol.~4, no.~1, pp.
  725--734, 1989.

\bibitem{BS-JJ-OA:09}
B.~Stott, J.~Jardim, and O.~Alsac, ``{DC} power flow revisited,'' \emph{IEEE
  Transactions on Power Systems}, vol.~24, no.~3, pp. 1290--1300, 2009.

\bibitem{jwsp-fd-fb:14c}
J.~W. Simpson-Porco, F.~D{\"o}rfler, and F.~Bullo, ``Voltage collapse in
  complex power grids,'' \emph{Nature Communications}, vol.~7, no. 10790, 2016.

\bibitem{NB:97}
N.~Biggs, ``Algebraic potential theory on graphs,'' \emph{Bulletin of the
  London Mathematical Society}, vol.~29, no.~6, pp. 641--683, 1997.

\bibitem{LOC-CAD-EDK:87}
L.~O. Chua, C.~A. Desoer, and E.~S. Kuh, \emph{Linear and Nonlinear
  Circuits}.\hskip 1em plus 0.5em minus 0.4em\relax McGraw-Hill, 1987.

\bibitem{PK:94}
P.~Kundur, \emph{Power System Stability and Control}.\hskip 1em plus 0.5em
  minus 0.4em\relax McGraw-Hill, 1994.

\bibitem{RK-FFW:84}
R.~J. Kaye and F.~F. Wu, ``Analysis of linearized decoupled power flow
  approximations for steady-state security assessment,'' \emph{IEEE
  Transactions on Circuits and Systems}, vol.~31, no.~7, pp. 623--636, 1984.

\bibitem{JT-DS-MIS:86}
J.~Thorp, D.~Schulz, and M.~Ili{\'c}-Spong, ``Reactive power-voltage problem:
  conditions for the existence of solution and localized disturbance
  propagation,'' \emph{International Journal of Electrical Power \& Energy
  Systems}, vol.~8, no.~2, pp. 66--74, 1986.

\bibitem{RAH-CRJ:85}
R.~A. Horn and C.~R. Johnson, \emph{Matrix Analysis}.\hskip 1em plus 0.5em
  minus 0.4em\relax Cambridge University Press, 1985.

\bibitem{JVM-KY-SMV-SZD-LMK:13}
J.~V. Milanovic, K.~Yamashita, S.~M. Villanueva, S.~Z. Djokic, and L.~M.
  Korunovic, ``International industry practice on power system load modeling,''
  \emph{IEEE Transactions on Power Systems}, vol.~28, no.~3, pp. 3038--3046,
  2013.

\bibitem{FD-FB:11d}
F.~D{\"o}rfler and F.~Bullo, ``Kron reduction of graphs with applications to
  electrical networks,'' \emph{IEEE Transactions on Circuits and Systems~I:
  Regular Papers}, vol.~60, no.~1, pp. 150--163, 2013.

\bibitem{TK-CL-KM-DM-RR-TU-KAZ:09}
T.~Kavitha, C.~Liebchen, K.~Mehlhorn, D.~Michail, R.~Rizzi, T.~Ueckerdt, and
  K.~A. Zweig, ``Cycle bases in graphs characterization, algorithms,
  complexity, and applications,'' \emph{Computer Science Review}, vol.~3,
  no.~4, pp. 199 -- 243, 2009.

\bibitem{BK-JV:08}
B.~Korte and J.~Vygen, \emph{Combinatorial Optimization: Theory and
  Algorithms}.\hskip 1em plus 0.5em minus 0.4em\relax Springer, 2008.

\bibitem{FD:13}
F.~D{\"o}rfler, ``Dynamics and control in power grids and complex oscillator
  networks,'' Ph.D. dissertation, Mechanical Engineering Department, University
  of California at Santa Barbara, Sep. 2013.

\bibitem{JWSP:17b}
J.~W. Simpson-Porco, ``A theory of solvability for lossless power flow
  equations -- {P}art {II}: Conditions for radial networks,'' \emph{IEEE
  Transactions on Control of Network Systems}, 2017.

\bibitem{BG-JWSP-FD-SZ-FB:13zb}
B.~Gentile, J.~W. Simpson-Porco, F.~D\"orfler, S.~Zampieri, and F.~Bullo, ``On
  reactive power flow and voltage stability in microgrids,'' in
  \emph{{A}merican {C}ontrol {C}onference}, Portland, OR, USA, Jun. 2014, pp.
  759--764.

\bibitem{RDZ-CEM-DG:11}
R.~D. Zimmerman, C.~E. Murillo-S{\'a}nchez, and R.~J. Thomas, ``{MATPOWER}:
  Steady-state operations, planning, and analysis tools for power systems
  research and education,'' \emph{IEEE Transactions on Power Systems}, vol.~26,
  no.~1, pp. 12--19, 2011.

\bibitem{CZ-SF-JM-PP:16}
C.~Josz, S.~Fliscounakis, J.~Maeght, and P.~Panciatici, ``{AC} power flow data
  in {MATPOWER} and {QCQP} format: i{T}esla, {RTE} snapshots, and {PEGASE},''
  2016, https://arxiv.org/abs/1603.01533.

\bibitem{AB-RJP:94}
A.~Berman and R.~J. Plemmons, \emph{Nonnegative Matrices in the Mathematical
  Sciences}.\hskip 1em plus 0.5em minus 0.4em\relax SIAM, 1994.

\end{thebibliography}

\appendices

\section{Supporting Results and Proofs}

\begin{pfof}{Proposition \ref{Prop:VLstar}}
That the definition \eqref{Eq:VLstar} is well-posed follows from Assumption \ref{Ass:Matrix}. From Assumption \ref{Ass:Matrix} and Fact \ref{Fact:Susceptance}, we conclude that $-B_{LL}$ is a symmetric positive definite $M$-matrix.\footnote{A matrix $A \in \real^{n \times n}$ is a \emph{$Z$-matrix} if $A_{ij} \leq 0$ for all $i\neq j$. A $Z$-matrix $A$ is a nonsingular \emph{$M$-matrix} if it can be expressed as $A = sI_n - B$, where $B \in \real^{n \times n}$ has nonnegative elements and $s > \rho(B)$, where $\rho(B)$ is the spectral radius of $B$ \cite[Chapter 6]{AB-RJP:94}. If $A$ is a nonsingular $M$-matrix, then the elements of $A^{-1}$ are nonnegative \cite[Chapter 6, Theorem 2.3, $\mathbf{N}_{38}$]{AB-RJP:94}.} Hence $-B_{LL}^{-1}$ is nonnegative. The submatrix $B_{LG}$ is nonnegative by Fact \ref{Fact:Susceptance}, and $V_G$ is strictly positive. Since the network is connected, $B_{LG}$ contains at least one non-zero positive element in each row and column, and it follows that $V_L^*$ is strictly positive. 
Now suppose that $P = \vzeros[m+m]$ in \eqref{Eq:Active} and $Q_L = \vzeros[n]$ in \eqref{Eq:Reactive}. Using Lemma \ref{Eq:DRPFE3}, substituting $\theta = \vzeros[n+m]$ into \eqref{Eq:AltRPFE}, and applying Lemma \ref{Eq:DRPFE3}(iii), it follows that $(\theta,V_L) = (\vzeros[n+m],V_L^*)$ is a solution of \eqref{Eq:Reactive}. Similarly, substituting $\theta = \vzeros[n+m]$ into \eqref{Eq:Active} shows that $(\vzeros[n+m],V_L^*)$ also solves \eqref{Eq:Active}. This completes the proof.
\end{pfof}

%

%




\medskip

\begin{lemma}[\bf Partitioned Incidence Matrix II]\label{Lem:IncidenceMatrixProperties2}
Consider the incidence matrix \eqref{Eq:Incidence} along with its plus/minus decomposition \eqref{Eq:IncidencePlusMinus}. For $v \in \real_{>0}^n$, let $v^{-1} = (v_1^{-1},\ldots,v_{n}^{-1})^{\sf T}$. The following identities hold:
\begin{enumerate}
\item[(i)] $[A_L^{g\ell}(-)^{\sf T}v]^{-1}\vones[g\ell] = A_L^{g\ell}(-)^{\sf T}v^{-1} = (|A_L^{g\ell}|)^{\sf T}v^{-1}$
\item[(ii)] $[A_L^{\ell\ell}(-)^{\sf T}v]^{-1}\vones[\ell\ell] = A_L^{\ell\ell}(-)^{\sf T}v^{-1}$
\item[(iii)] $[A_L^{\ell\ell}(+)^{\sf T}v]^{-1}\vones[\ell\ell] = A_L^{\ell\ell}(+)^{\sf T}v^{-1}$
\item[(iv)] $[v]^{-1}A_L^{\ell\ell}(+) = A_L^{\ell\ell}(+)[A_L^{\ell\ell}(+)^{\sf T}v^{-1}]$
\item[(v)] $[v]^{-1}A_L^{\ell\ell}(-) = A_L^{\ell\ell}(-)[A_L^{\ell\ell}(-)^{\sf T}v^{-1}]$
\item[(vi)] $[v]^{-1}A_L^{g\ell}(-) = A_L^{g\ell}(-)[A_L^{g\ell}(-)^{\sf T}v^{-1}]$\,.
\end{enumerate}
\end{lemma}

\begin{proof}.
{\tb 
(i) By construction $A_{L}^{g\ell}(-) \in \{0,1\}^{n \times |\Egl|}$ has exactly one element equal to one in each column; if the column corresponds to branch $(i,j)\in \Egl$, then the non-zero element is in row $j$. Thus, $A_{L}^{g\ell}(-)^{\sf T}v^{-1}$ produces a branch vector with entry $v_j^{-1}$ in the entry corresponding to branch $(i,j)\in\Egl$. It follows that $[A_L^{g\ell}(-)^{\sf T}v][A_L^{g\ell}(-)^{\sf T}v^{-1}] = I_{g\ell}$, and the result follows by right-multiplying by $\vones[\ell\ell]$ and rearranging. The third equality is trivial as $A_{L}^{g\ell}(-) = |A|_L^{g\ell}$ by definition. The proofs of (ii)-(v) are analogous.
}
\end{proof}


\medskip

\begin{lemma}[\bf Identities]\label{Lem:Identities}
The following identities hold:
\begin{enumerate}[(i)]\setlength{\itemsep}{5pt}
\item \label{Lem:Identity1}
$
\!
\begin{aligned}[t]
 L_1(v) 
    &\triangleq [v]^{-1}|A|_L^{\ell\ell}[h_{\ell\ell}(v)]\\
    &= A_L^{\ell\ell}(+)[A_L^{\ell\ell}(-)^{\sf T}v] + A_L^{\ell\ell}(-) [A_L^{\ell\ell}(+)^{\sf T}v]\,,
\end{aligned}
$ 
\item \label{Lem:Identity2} $L_2(v) \triangleq [v]^{-1}|A|_L^{g\ell}[h_{g\ell}(v)] = |A|_L^{g\ell}$\,.
\end{enumerate}
\end{lemma}

\begin{proof}.
(i): Using the decomposition $|A|_L^{\ell\ell} = A_{L}^{\ell\ell}(+) + A_{L}^{\ell\ell}(-)$ from   \eqref{Eq:IncidenceAbs} and inserting the expression for $h_{\ell\ell}(v)$ from  \eqref{Eq:VVectorForm}, we find that
\begin{align*}
L_1(v) &= [v]^{-1}A_L^{\ell\ell}(+)[A_L^{\ell\ell}(+)^{\sf T}v] [A_L^{\ell\ell}(-)^{\sf T}v]\\
&+ [v]^{-1}A_L^{\ell\ell}(-)[A_L^{\ell\ell}(+)^{\sf T}v] [A_L^{\ell\ell}(-)^{\sf T}v]
\end{align*}
Using Lemma \ref{Lem:IncidenceMatrixProperties2} (iv) and (v) to substitute for the first factor in each term, then reshuffling the diagonal matrices, we find that
\begin{align*}
L_1(v) &= A_L^{\ell\ell}(+)\cdot\underbrace{[A_L^{\ell\ell}(+)^{\sf T}v^{-1}]\cdot[A_L^{\ell\ell}(+)^{\sf T}v]}_{=I_{\ell\ell}\,\mathrm{by}\,\,\mathrm{Lemma}\,\,\ref{Lem:IncidenceMatrixProperties2}\,\mathrm{(iii)}}\cdot[A_L^{\ell\ell}(-)^{\sf T}v]
\\
&+ A_L^{\ell\ell}(-)\cdot \underbrace{[A_L^{\ell\ell}(-)^{\sf T}v^{-1}]\cdot [A_L^{\ell\ell}(-)^{\sf T}v]}_{=I_{\ell\ell}\,\mathrm{by}\,\,\mathrm{Lemma}\,\,\ref{Lem:IncidenceMatrixProperties2}\,\mathrm{(ii)}}\cdot\, [A_L^{\ell\ell}(+)^{\sf T}v]
\end{align*}
from which the result follows.

(ii): Similar to (i), first substitute for $|A|_{L}^{g\ell}$ from \eqref{Eq:IncidenceAbs} and for $h_{g\ell}(v)$ from \eqref{Eq:VVectorForm} to find
$$
L_2(v) = [v]^{-1}A_L^{g\ell}(-)[A_L^{g\ell}(-)^{\sf T}v]\,.
$$
Applying Lemma \ref{Lem:Identities} (vi) to the first two terms in the product, we obtain
\begin{align*}
L_2(v) &= A_L^{g\ell}(-)\underbrace{[A_L^{g\ell}(-)^{\sf T}v^{-1}][A_L^{g\ell}(-)^{\sf T}v]}_{{=I_{g\ell}\,\mathrm{by}\,\,\mathrm{Lemma}\,\,\ref{Lem:IncidenceMatrixProperties2}\,\mathrm{(i)}}} = |A|_L^{g\ell}\,,
\end{align*}
which completes the proof.
\end{proof}


\medskip

\begin{lemma}\label{Lem:EquivFormulas}\textbf{(Alternate Expressions for the Reactive Power Flow Equation):}
The reactive power flow \eqref{Eq:Reactive} can be written in vector form as
\begin{equation}\label{Eq:AltRPFE}
\begin{aligned}
Q_L &= -[V_L][B_{ii}]_{i\in\mathcal{N}_L}V_L\\ &\quad - |A|_L[V_iV_jB_{ij}]_{(i,j)\in\mathcal{E}}\,\bcos(A^{\sf T}\theta)\,.
\end{aligned}
\end{equation}
Moreover, the following expressions are all equal:
\begin{enumerate}
\item[(i)]\label{Eq:DRPFE1} $-[V_L][B_{ii}]_{i\in\mathcal{N}_L}V_L - |A|_L[V_iV_jB_{ij}]_{(i,j)\in\mathcal{E}}\vones[|\mathcal{E}|]$\,,
\item[(ii)]\label{Eq:DRPFE2} $-[V_L]\left(B_{LL} V_L +B_{LG} V_G\right)$\,,
\item[(iii)]\label{Eq:DRPFE3} $-[V_L]B_{LL}\left(V_L-V_L^*\right)$\,.
\end{enumerate}
\end{lemma}

\begin{proof}.
{\tb
For $i \in \mathcal{N}_L$, the $i$th component of \eqref{Eq:AltRPFE} is given by
\begin{align*}
Q_i &= -V_i^2B_{ii} - \sum_{\substack{(i,j)\in\mathcal{E} \\ (j,i)\in\mathcal{E}}} B_{ij}V_iV_j\cos(\theta_i-\theta_j)\\ &= -\sum_{j=1}^{n+m}\nolimits B_{ij}V_iV_j\cos(\theta_i-\theta_j)\,,
\end{align*}
which is exactly \eqref{Eq:Reactive}. To obtain (i), simply set $A^{\sf T}\theta = \vzeros[|\mathcal{E}|]$ in \eqref{Eq:AltRPFE}. The equivalence of (i) and (ii) follows most easily by observing that both are equivalent to \eqref{Eq:Reactive} by setting $\cos(\theta_i-\theta_j) = 1$ in \eqref{Eq:Reactive} and separating the sum in the appropriate fashion to obtain (i) or (ii). The equivalence of (ii) and (iii) is immediate by invoking Assumption \ref{Ass:Matrix} and inserting \eqref{Eq:VLstar}.
}
\end{proof}

\begin{pfof}{Corollary \ref{Cor:FPPFAcyclicAlt}}
Let $k(v)$ denote the final term in \eqref{Eq:f}. Expanding $k(v)$ by substituting the block incidence matrix \eqref{Eq:Incidence}, the block $\mathsf{D}$ matrix \eqref{Eq:DMatrix2}, and the partitioned $h(v)$ from \eqref{Eq:VVectorForm}, we obtain
\begin{align*}
k(v) &= \frac{1}{4}\mathsf{S}^{-1}\underbrace{[v]^{-1}|A|_L^{\ell\ell} [h_{\ell\ell}(v)]}_{\triangleq L_1(v)}\mathsf{D}_{\ell\ell}u_{\ell\ell}(v)\\
&\quad + \frac{1}{4}\mathsf{S}^{-1}\underbrace{[v]^{-1}|A|_L^{g\ell} [h_{g\ell}(v)]}_{\triangleq L_2(v)}\mathsf{D}_{g\ell}u_{g\ell}(v)\,,
\end{align*}
where we have rearranged some diagonal matrices, and where by combining \eqref{Eq:ux} and \eqref{Eq:psix}, 
\begin{subequations}
\begin{align*}
u_{\ell\ell}(v) &= \vones[\ell\ell] - \sqrt{\vones[\ell\ell] - [h_{\ell\ell}(v)]^{-2}\mathsf{D}_{\ell\ell}^{-2}[p_{\ell\ell}]p_{\ell\ell}}\,,\\
u_{g\ell}(v) &= \vones[g\ell] - \sqrt{\vones[g\ell] - [h_{g\ell}(v)]^{-2}\mathsf{D}_{g\ell}^{-2}[p_{g\ell}]p_{g\ell}}\,.
\end{align*}
\end{subequations}
Lemma \ref{Lem:Identities}(\ref{Lem:Identity2}) shows that $L_2(v) = |A|_L^{g\ell}$ independent of $v$, so the previous simplifies to 
\begin{align*}
k(v) &= \frac{1}{4}\mathsf{S}^{-1}L_1(v)\mathsf{D}_{\ell\ell}u_{\ell\ell}(v) + \frac{1}{4}\mathsf{S}^{-1}|A|_L^{g\ell}\mathsf{D}_{g\ell}u_{g\ell}(v)\,.
\end{align*}
Applying Lemma \ref{Lem:Identities}(\ref{Lem:Identity1}) to explicitly expand $L_1(v)$ leads immediately to \eqref{Eq:FixedPoint}.
\end{pfof}

%

\vspace{-2em}

\begin{IEEEbiography}[{\includegraphics[width=1in,height=1.25in,clip,keepaspectratio]{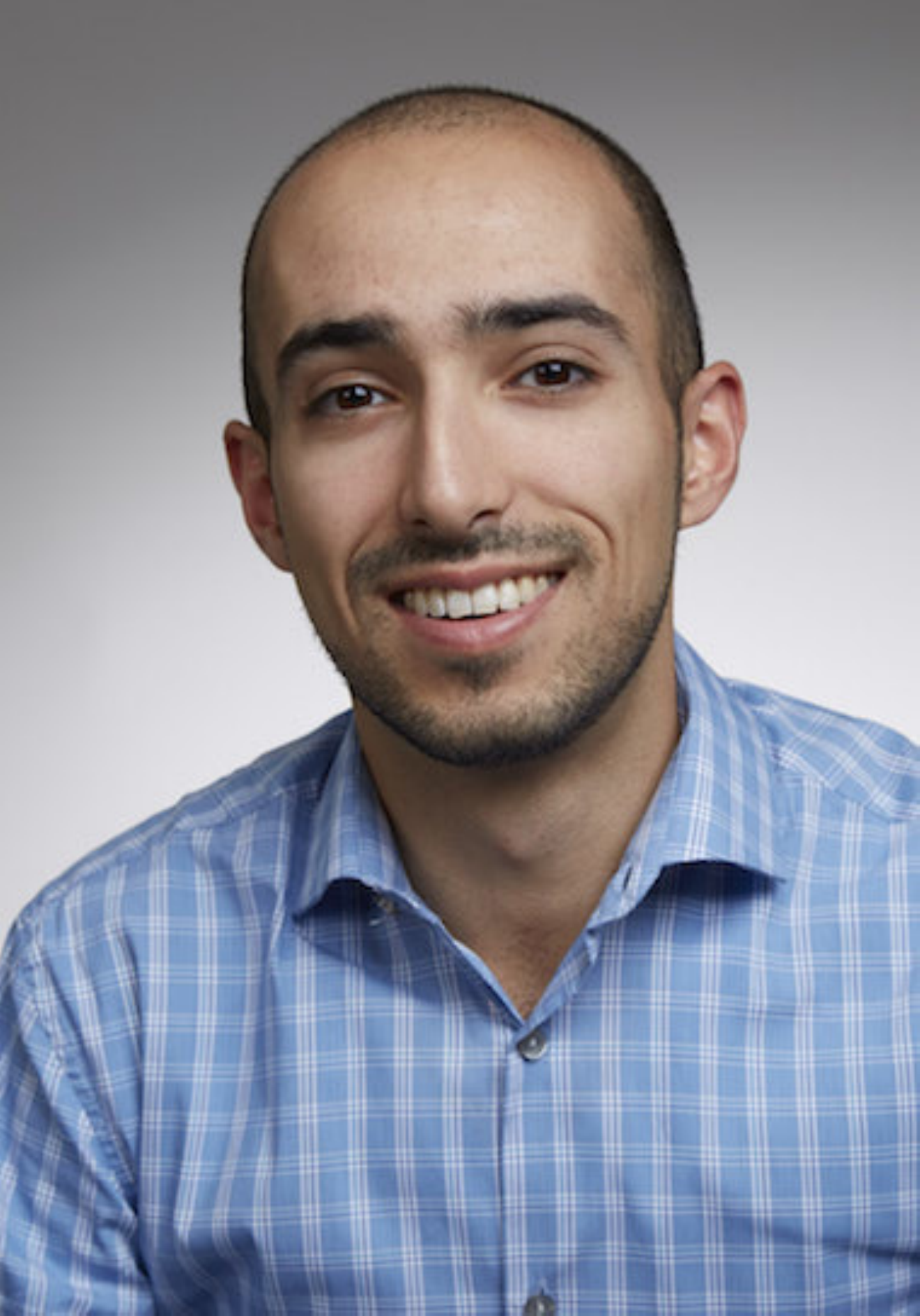}}]{John W. Simpson-Porco} (S'11--M'16) received the B.Sc. degree in engineering physics from Queen's University, Kingston, ON, Canada in 2010, and the Ph.D. degree in mechanical engineering from the University of California at Santa Barbara, Santa Barbara, CA, USA in 2015.

He is currently an Assistant Professor of Electrical and Computer Engineering at the University of Waterloo, Waterloo, ON, Canada. He was previously a visiting scientist with the Automatic Control Laboratory at ETH Z\"{u}rich, Z\"{u}rich, Switzerland. His research focuses on the control and optimization of multi-agent systems and networks, with applications in modernized power grids.

Prof. Simpson-Porco is a recipient of the 2012--2014 IFAC Automatica Prize and the Center for Control, Dynamical Systems and Computation Best Thesis Award and Outstanding Scholar Fellowship.
\end{IEEEbiography}
%
%
%

\end{document}